\documentclass[12pt,leqno]{amsart}

\usepackage{enumerate}
\usepackage{graphicx}
\usepackage{epsfig}
\usepackage{amssymb,amsmath,amsthm,amsfonts}

\usepackage {latexsym}

\usepackage{bbm}

\allowdisplaybreaks

\newcommand{\nc}{\newcommand}
\nc{\les}{\lesssim}
\nc{\nit}{\noindent}
\nc{\nn}{\nonumber}
\nc{\D}{\partial}
\nc{\diff}[2]{\frac{d #1}{d #2}}
\nc{\diffn}[3]{\frac{d^{#3} #1}{d {#2}^{#3}}}
\nc{\pdiff}[2]{\frac{\partial #1}{\partial #2}}
\nc{\pdiffn}[3]{\frac{\partial^{#3} #1}{\partial{#2}^{#3}}}
\nc{\abs}[1] {\lvert #1 \rvert}
\nc{\cAc}{{\cal A}_c}
\nc{\cE}{{\cal E}}
\nc{\cF}{{\cal F}}
\nc{\cP}{{\cal P}}
\nc{\cV}{{\cal V}}
\nc{\cQ}{{\cal Q}}
\nc{\cGin}{{\cal G}_{\rm in}}
\nc{\cGout}{{\cal G}_{\rm out}}
\nc{\cO}{{\cal O}}
\nc{\Lav}{{\cal L}_{\rm av}}
\nc{\cL}{{\cal L}}
\nc{\cB}{{\cal B}}
\nc{\cZ}{{\cal Z}}
\nc{\cR}{{\cal R}}
\nc{\cT}{{\cal T}}
\nc{\cY}{{\cal Y}}
\nc{\cX}{{\cal X}}
\nc{\cXT}{{{\cal X}(T)}}
\nc{\cBT}{{{\cal B}(T)}}
\nc{\vD}{{\vec \mathcal{D}}}
\nc{\efield}{\mathcal{E}}
\nc{\vE}{{\vec \efield}}
\nc{\vB}{{\vec \mathcal{B}}}
\nc{\vH}{{\vec \mathcal{H}}}
\nc{\F}{  \mathcal{F} }
\nc{\ty}{{\tilde y}}
\nc{\tu}{{\tilde u}}
\nc{\tV}{{\tilde V}}
\nc{\Pc}{{\bf P_c}}
\nc{\bx}{{\bf x}}
\nc{\bX}{{\bf X}}
\nc{\bXYZ}{{\bf XYZ}}
\nc{\bY}{{\bf Y}}
\nc{\bF}{{\bf F}}
\nc{\bS}{{\bf S}}
\nc{\dV}{{\delta V}}
\nc{\dE}{{\delta E}}
\nc{\TT}{{\Theta}}
\nc{\dPsi}{{\delta\Psi}}
\nc{\order}{{\cal O}}
\nc{\Rout}{R_{\rm out}}
\nc{\eplus}{e_+}
\nc{\eminus}{e_-}
\nc{\epm}{e_\pm}
\nc{\eps}{\varepsilon}
\nc{\vnabla}{{\vec\nabla}}
\nc{\G}{\Gamma}
\nc{\w}{\omega}
\nc{\mh}{h}
\nc{\mg}{g}
\nc{\vphi}{\varphi}
\nc{\tlambda}{\tilde\lambda}
\nc{\be}{\begin{equation}}
\nc{\ee}{\end{equation}}
\nc{\ba}{\begin{eqnarray}}
\nc{\ea}{\end{eqnarray}}

\nc{\g}{\gamma}
\nc{\ol}{\overline}

\newtheorem{theorem}{Theorem}[section]
\newtheorem{lemma}[theorem]{Lemma}
\newtheorem{prop}[theorem]{Proposition}
\newtheorem{corollary}[theorem]{Corollary}
\newtheorem{defin}[theorem]{Definition}
\newtheorem{rmk}[theorem]{Remark}

\nc{\pT}{\partial_T}
\nc{\pz}{\partial_z}
\nc{\pt}{\partial_t}
\nc{\la}{\langle}
\nc{\ra}{\rangle}
\nc{\infint}{\int_{-\infty}^{\infty}}
\nc{\halfwidth}{6.5cm}
\nc{\figwidth}{10cm}
\newcommand{\f}{\frac}

\nc{\nlayers}{L} \nc{\nsectors}{M}
\nc{\indicator}{\mathbf{1}}
\nc{\Rhole}{R_{\rm hole}}
\nc{\Rring}{R_{\rm ring}}
\nc{\neff}{n_{\rm eff}}
\nc{\Frem}{F_{\rm rem}}
\nc{\R}{\mathbb R}
\nc{\C}{\mathbb C}
\nc{\Z}{\mathbb Z}
\nc{\DD}{\Delta}
\nc{\cD}{\mathcal D}
\nc{\lnorm}{\left\|}
\nc{\rnorm}{\right\|}
\nc{\rnormp}{\right\|_{\ell^{p,\eps}}}
\nc{\rar}{\rightarrow}
\nc{\mR}{\mathcal R}
\nc{\oo}{\"o}   
\nc{\os}{\overset{o}}
\begin{document}

\begin{abstract}
We study the massive two dimensional Dirac operator with an electric potential. In particular, we show that 
the $t^{-1}$ decay rate holds in the $L^1\to L^\infty$ setting if the threshold energies are regular.  We also show these bounds hold in the presence of s-wave resonances at the threshold.  We further show that, if the threshold energies are regular that a faster decay rate of $t^{-1}(\log t)^{-2}$ is attained for large $t$, at the cost of logarithmic spatial weights. The free Dirac equation does not satisfy this bound due to the s-wave resonances at the threshold energies.
\end{abstract}

\title[Dispersive estimates for Dirac Operators]{ Dispersive estimates for massive Dirac operators in dimension two}

\author[Erdogan, Green, Toprak]{M. Burak Erdo\u{g}an, William~R. Green and Ebru Toprak}
\thanks{The first and third authors were partially supported by DMS-1501041. The third author acknowledges the support of a Rose-Hulman summer professional development grant.}
\address{Department of Mathematics \\
University of Illinois \\
Urbana, IL 61801, U.S.A.}
\email{berdogan@illinois.edu}
\address{Department of Mathematics\\
Rose-Hulman Institute of Technology \\
Terre Haute, IN 47803, U.S.A.}
\email{green@rose-hulman.edu}
\address{Department of Mathematics \\
University of Illinois \\
Urbana, IL 61801, U.S.A.}
\email{toprak2@illinois.edu}

\maketitle

\section{Introduction}
We consider the linear Dirac equation with potential,
\begin{align}\label{dirac}
i\partial_t \psi(x,t) = (D_m +V(x)) \psi(x,t), \,\,\,\, \psi(x,0)= \psi_0(x),
\end{align} 
with $\psi(x,t) \in \mathbb{C}^{2}$ when the spatial variable $(x_1,x_2)= x \in \R^2$. The free Dirac operator $D_m$ is defined by 
$$
	D_m=-i\alpha\cdot \nabla +m\beta = -i \alpha_1 \partial_{x_1}-i\alpha_2 \partial_{x_2}+m\beta.
$$
Here $m\geq0$ is the mass of the quantum particle.  When $m>0$, \eqref{dirac} is the massive Dirac equation and when $m=0$, the equation is massless.  The $2\times 2$ Hermitian matrices $\alpha_j$ (with $\alpha_0=\beta$) satisfy the anti-commutation relationship
\be\label{eqn:anticomm}
	\alpha_j\alpha_k+\alpha_k \alpha_2 =2 \delta_{jk} \mathbbm 1_{\mathbb C^2},
	\qquad 0\leq j,k\leq 2,
\ee
By convention, we take
\begin{align}
\beta=\left(\begin{array}{cc} 1& 0\\ 0 & -1
\end{array}\right), \qquad \alpha_1=\left(\begin{array}{cc} 0 & 1\\ 1 & 0
\end{array}\right), \qquad
\alpha_2=\left(\begin{array}{cc} 0 & -i\\ i & 0
\end{array}\right).
\end{align}
The hyperbolic system \eqref{dirac} was derived by Dirac to describe the evolution of a quantum particle at near luminal speeds.  We view the system as a relativistic modification of the Schr\"odinger equation.  This viewpoint is fruitful in light of the following identity\footnote{When we write scalar operators such as  $-\Delta+m^2-\lambda^2$, they are to be understood as $(-\Delta+m^2-\lambda^2)\mathbbm 1_{\mathbb C^{2}}$.  Similarly, we write $L^p$ to indicate $L^p(\R^2) \times L^p(\R^2)$.},  which follows from   \eqref{eqn:anticomm}:
\be  \label{dirac_schro_free}
(D_m-\lambda \mathbbm 1)(D_m+\lambda \mathbbm 1) =(-i\alpha\cdot \nabla +m\beta -\lambda \mathbbm 1)
(-i\alpha\cdot \nabla+m\beta+\lambda \mathbbm 1)   =(-\Delta+m^2-\lambda^2). 
\ee
This yields the  identity 
\be \label{eq:Dmpm}
\mR_0(\lambda)=(D_m+\lambda)R_0(\lambda^2-m^2)
\ee
for  the free Dirac resolvent,   $\mR_0(\lambda)=(D_m-\lambda)^{-1}$, where $R_0(\lambda)=(-\Delta-\lambda)^{-1}$ is the Schr\"odinger free resolvent  and $\lambda$ is in the resolvent set.
We refer the reader to the text of Thaller, \cite{Thaller}, for a more extensive introduction to the Dirac equation.

Our goal in this paper is to put the dispersive estimates for the massive Dirac equation on the same ground as those for the Schr\"odinger equation, \cite{Sc2,eg2,eg3}.  For the remainder of the paper $m>0$.  To this end, we extend the recent results of the first two authors, \cite{egd}, in two significant ways.  First, we show that the dispersive bounds hold uniformly, that is we show that the $ H^1 \to BMO$ bounds in \cite{egd} remain valid as operators from $L^1 \to L^\infty$.  Second, we show a large time integrable bound holds at the cost of spatial weights.  To state our results, we employ the following notation. Let $P_{ac}(H)$ be the projection on the absolutely continuous spectral subspace of $L^2(\R^2)$ associated with $H$.  In addition, we define $a-:=a-\epsilon$ for a small, but fixed $\epsilon>0$.

Our main result is the following logarithmically weighted decay estimate with an integrable decay rate in $t$:
\begin{theorem}  \label{th:main} 
	Assume that the matrix valued potential $V(x)$ is self-adjoint, with continuous entries satisfying $|V_{ij}(x)| \les \la x \ra ^{-\delta}$ for $\delta > 5 $. If the threshold energies  $\pm m$ are regular then we have  
	\begin{align}
		\| w^{-2} e^{-itH} P_{ac}(H)  \la H \ra ^{-\f 72-} f \|_{L^{\infty}(\R^2)} \les \frac{1}{ |t| \log^2 |t|} \| w^2 f \|_{L^1(\R^2)}, \,\,\ |t|>2
	\end{align} 
	where $w(x) = 1+ \log^{+} |x| $. 
\end{theorem} 
It is worth noting that the free Dirac equation does not satisfy this estimate due to the  s-wave  resonances at the threshold energies. In particular, in the proof of  Theorem~\ref{th:main} we encounter several terms behaving like $\f1t$ for large $t$ including one  term coming from the free part in the Born series. However, in the regular case these terms cancel each other in pairs. Our earlier results in \cite{eg3,ebru} on the  Schr\"odinger's equation rely on similar observations.   

We also have the following  global decay estimate:
\begin{theorem}\label{th:main1} 
	Assume that the matrix valued potential $V(x)$ is self-adjoint, with continuous entries satisfying $|V_{ij}(x)| \les \la x \ra ^{-\delta}$ for $\delta > 3 $. If the threshold energies $\pm m$ are regular, or if there are s-wave resonances only, then the kernel of the solution operator satisfies
	$$
		\sup_{x,y\in \R^2} \big|[e^{-itH} P_{ac}(H) \la H \ra ^{-\f 72 -} ](x,y) \big| \les \la t\ra ^{-1}. 
	$$
	As a consequence, we obtain the mapping estimate
	$$
		\big\| [e^{-itH} P_{ac}(H) \la H \ra ^{-\f72 -} f\big\|_{L^\infty(\R^2)} \les \la t\ra^{-1} \| f\|_{L^1(\R^2)}.
	$$
\end{theorem} 
Finally we state a polynomially weighted estimate:  
\begin{theorem} \label{th:main2} 
	Assume that the matrix valued potential $V(x)$ is self-adjoint, with continuous entries satisfying $|V_{ij}(x)| \les \la x \ra ^{-\delta}$ for $\delta > 5 $. If the threshold energies  $\pm m$ are regular the kernel of the perturbed solution operator satisfies 
	$$
		\big|[e^{-itH} P_{ac}(H) \la H \ra ^{-3-} ](x,y) \big| \les \frac{w(x) w(y)}{t \log^2 (t)} + \frac{ \la x \ra ^{\f 3 2} \la y \ra ^{\f 3 2}}{t^{1+}}, \,\,\ t>2.
	$$
\end{theorem} 
As in \cite{eg3} and \cite{ebru}, to obtain Theorem~\ref{th:main} we interpolate the results of Theorems~\ref{th:main1} and \ref{th:main2} using
$$ 
	\min\bigg(1, \frac{a}{b}\bigg) =  \frac{\log^2 a}{\log^2 b}, \,\,\ a,b>2.
$$
We note that the continuity assumption on the entries of the potential is needed only in the large energy regime to use the limiting absorption principle in \cite{egg}.  The loss of derivatives on the initial data embodied in the negative powers of $H$ are also a high energy issue.  

We prove our dispersive estimates by considering the Dirac solution operator as an element of the functional calculus.  Specifically, we employ the Stone's formula to see
\be\label{eq:Stone}
	e^{-itH}P_{ac}(H)f =\frac{1}{2\pi i} \int_{\sigma_{ac}(H)} e^{-it\lambda} [\mR_V^+(\lambda)-\mR_V^-(\lambda) ] f  \, d\lambda.
\ee
Here the difference of the perturbed resolvents $\mR_V^\pm(\lambda)=\lim_{\epsilon\to 0^+} (D_m+V-(\lambda \pm i\epsilon))^{-1}$ provides the spectral measure.  These operators are well defined between weighted $L^2$ spaces by the limiting absorption principle, \cite{agmon,BaHe,GM,egg}.
 
The literature on the perturbed Dirac equation is smaller than that on other dispersive equations such as the Schr\"odinger, wave and Klein-Gordon.  D'Ancona and Fanelli \cite{DF} were the first, to the authors' knowledge, to study the pointwise time decay for the perturbed Dirac evolution.  They studied the three dimensional massless Dirac equation and related wave equations with small electromagnetic potentials.  Escobedo and Vega, \cite{EV}, established dispersive and Strichartz estimates for the three dimensional free Dirac equation to study a semi-linear Dirac equation.  Boussaid, \cite{Bouss1}, proved dispersive estimates on Besov spaces, and  on weighted $L^2$ spaces for the massive three dimensional Dirac equation, and applied these estimates to studying  `particle-like solutions' for a class of non-linear Dirac equations.  See also the recent works of Boussaid and Comech on non-linear Dirac equations, \cite{BC1,BC2}.  

The existence of threshold resonances or eigenvalues are known to affect the dispersive estimates in the case of the Schr\"odinger evolution, \cite{JenKat,Mur,ES,Yaj3,ES2,eg2,Bec}.  The effect of threshold obstructions for the massive three dimensional Dirac equation was studied by the authors in \cite{EGT}.  The threshold resonance structure is more complicated in the two dimensional case; only the effect of the `s-wave' resonance on the dispersive estimates has been established, see \cite{egd} and Theorem~\ref{th:main1} above.

Smoothing and/or Strichartz estimates for the Dirac equation have been established by various authors, see  \cite{BDF,C,CS,egd,egg} for example.  In the two dimensional case, Kopylova considered estimates on weighted $L^2$ spaces, \cite{kopy}, which had roots in the work of Murata, \cite{Mur}.  In \cite{BH}, Bejenaru and Herr obtained frequency-localized  estimates for the free equation   in two dimensions  to study the cubic non-linear Dirac equation. Dispersive estimates for one-dimensional Dirac equation was considered in \cite{CTS}. 

Our approach relies on a detailed analysis of the resolvent operators.  We follow the strategy employed to analyze the two-dimensional Schr\"odinger equation set out by Schlag in \cite{Sc2} and in our earlier works \cite{eg2,eg3,eg4,ebru,egd}.  Extending these results to other dispersive equations such as the wave equation is  non-trivial, see \cite{Gwave,egd}. 

We briefly recall some spectral theory for the Dirac operator.  For the class of potentials we consider, Weyl's criterion implies that the essential spectrum coincides  with the spectrum $(-\infty,-m]\cup [m,\infty)$ of the free operator \cite{Thaller}. There is no singular continuous spectrum \cite{BaHe,GM}, and no embedded eigenvalues in $(-\infty,-m)\cup (m,\infty)$ \cite{BG1,BC1}. In addition there can only be finitely many eigenvalues in the gap $[-m,m]$ \cite{GM,egd}.


To establish the high energy bound in Theorems~\ref{th:main1} and \ref{th:main2}, one also needs a limiting absorption principle  for the perturbed resolvent operator of the form:
\be\label{eqn:lap1}
\sup_{|\lambda|>\lambda_0}\|\partial_\lambda^k \mR_V^\pm (\lambda)\|_{L^{2,\sigma}\to L^{2,-\sigma}} \les 1,
\qquad \sigma > \f12 +k, \,\,\,\,k=0,1,2
\ee  
holds for any $\lambda_0>m$. Unlike  
Schr\"odinger, the Dirac resolvent  does not decay in the spectral parameter $\lambda$ as $\lambda\to \infty$ even for the free resolvent, \cite{Yam}.  As a consequence, Agmon's bootstrapping  argument \cite{agmon} does not suffice to establish \eqref{eqn:lap1}.  Instead, the argument may only be used to establish uniform bounds on compact subsets of the continuous spectrum, see e.g. \cite{Yam,GM}. Recently, the first two authors and Goldberg, \cite{egg}, showed
$$
\sup_{ |\lambda|>m}\| \mR_V^\pm (\lambda)\|_{L^{2,\sigma}\to L^{2,-\sigma}} \les 1, \qquad \sigma > \f12 
$$
in any dimension $n\geq 2$ in both the massive and massless cases.  Using standard arguments, one can easily show \eqref{eqn:lap1} from this bound. 
Other limiting absorption principles have been obtained, see  \cite{Bouss1,DF,BG}. Georgescu and Mantoiu \cite{GM} obtained a limiting absorption principle in general dimensions on compact subsets. 

A threshold s-wave resonance may be characterized in terms of distributional solutions to $H\psi=0$ with $\psi \in L^\infty \setminus L^p(\R^2)$ for any $p<\infty$, \cite{egd}.  Such a resonance is natural as the free Dirac operator $D_m$ has s-wave resonances at threshold energies, $\psi_m=(1,0)^T$ and $\psi_{-m}=(0,1)^T$ at $\lambda=\pm m$ respectively.  We show that the existence an s-wave resonance for the perturbed Dirac equation satisfies the same dispersive bounds as the free Dirac equation.  The effect of threshold p-wave resonances and/or eigenvalues on the dispersive estimates are still open.  

The paper is organized as follows.  We begin in Section~\ref{sec:resexp}  by developing the necessary low energy expansions for the Dirac resolvent operators.  In Section~\ref{sec:nonweight}, we prove Theorem~\ref{th:main1} by considering energies close to, and separated from, the thresholds respectively.  In Section~\ref{sec:weighted} we prove Theorem~\ref{th:main2}, again by considering the different energy regimes.  Finally in Section~\ref{sec:tech} we collect  technical integral estimates needed to prove Theorems~\ref{th:main1} and \ref{th:main2}.

 \section{Resolvent expansion around threshold}\label{sec:resexp}
 In this section we obtain expansions for the resolvent operator $\mR^{\pm}_V(\lambda)$ in a neighborhood of the threshold energies using the properties of free Schr\"odinger resolvent operator $R_0(z^2)=(-\Delta-z^2)^{-1}$.  The expansions for  $\mR^{\pm}_0(\lambda)$ and $\mR^{\pm}_V(\lambda)$ we give in this section has been established by first and second authors in \cite{egd}.  We provide a few modifications that will be needed in Section~\ref{sec:weighted}.
 
  First, we review some estimates (see e.g. \cite{Sc2,eg2,eg3}) for $R_0^\pm(z^2)$. Recall that in $\R^n$ the integral kernel of the free resolvent is given by Hankel functions, see \cite{JN}.
For $n=2$ we have
\begin{align}\label{R0 def}
\begin{split}
	R_0^\pm(z^2)(x,y)=\pm {\f i 4 } H_0^{\pm}( z |x-y|) = {\f 1 4 } \big[\pm i J_0(z|x-y| ) -Y_0(z |x-y| ) \big]. 
	     \end{split}
	     \end{align}
Here $J_0(u)$ and $Y_0(u)$ are  Bessel functions of the first and second kind of order zero. We use the notation $f=\widetilde{O}(g) $ to indicate 
  \begin{align} 
  \label{frak}
   \frac{d^j}{d\lambda ^j} f = O(\frac{d^j}{d\lambda^j} g ), \hspace{4mm} j=0,1,2,....,. 
     \end{align}
If (\ref{frak}) is satisfied only for $j=0,1,2,..,k$ we use the notation $f=\widetilde{O}_k(g). $

 For $|u|\ll1$, we have the series expansions for Bessel functions,   
\begin{align}
 	J_0(u)&=1-\frac{1}{4}u^2+\frac{1}{64}u^4+\widetilde O_6(u^6),\label{J0 def}\\
 	Y_0(u)&=\frac2\pi \log(u/2)+\f{2\gamma}{\pi}+ \widetilde O(u^2\log(u)) \label{Y0 def}.
               \end{align} 
For any $\mathcal C\in\{J_0, Y_0\}$ we also have the following representation if $ |u| \gtrsim1 $.
     \begin{align}
     \label{largr}
      	\mathcal C(z)=e^{iu} \omega_+(u)+e^{-iu}\omega_-(u), \qquad
     	\omega_{\pm}(u)=\widetilde O\big((1+|u|)^{-\frac{1}{2}}\big)
      \end{align}

Lemma~\ref{lem:schro_resol} follows  from these expansions.
\begin{lemma} \label{lem:schro_resol} 
	For $z|x-y|<1$, we have the expansion
	\begin{align}\label{eq:freschroresolv}
		R_0^{\pm}(z^2) =g^\pm(z)+G_0
		+  \widetilde 
		O_2(z^2|x-y|^2\log(z|x-y|))
			\end{align}
 where
\begin{align}
	g^\pm(z)&=-\frac1{2\pi} \big(\log(z/2)+\gamma\big)\pm\frac{i}4 \label{g def}\\
	G_0f(x)&=-\frac{1}{2\pi}\int_{\R^2} \log|x-y|f(y)\,dy, \label{G0 def}
\end{align}
\end{lemma}
The following analysis is performed on the positive portion  $[m,\infty)$ of the spectrum $H$. See Remark~\ref{rmk:resonances}-$i$ below for the negative branch, $(-\infty,-m]$. We write $\lambda=\sqrt{m^2+z^2}$ with 
$0<z\ll 1$. Using  \eqref{eq:Dmpm} we have
\begin{multline}\label{eq:dr1}
	\mR_0^\pm(\lambda)=\left[-i\alpha \cdot \nabla + m \beta +\sqrt{m^2+z^2} I\right]R_0^\pm(z^2) = \\  
	\left[-i\alpha \cdot \nabla + m (\beta +I)+\frac{z^2}{2m} I +\widetilde O(z^4) I\right]R_0^\pm(z^2).
\end{multline}
We now employ the following notational conventions. The operators  $M_{11}$ and $M_{22}$
are defined to be matrix-valued operators with kernels 
\begin{align*}
	M_{11}(x,y)=\left(\begin{array}{cc}
		1 & 0 \\ 0 & 0
	\end{array}\right), \qquad M_{22}(x,y)=\left(\begin{array}{cc}
		0 & 0 \\ 0 & 1
	\end{array}\right).
\end{align*}
We also define the projection operators $I_1, I_2$ by
\begin{align*}
	I_1 \left(\begin{array}{c}
		a \\ b\end{array}\right)=\left(\begin{array}{c}
		a \\0\end{array}\right),\qquad
	I_2 \left(\begin{array}{c}
		a \\ b\end{array}\right)=\left(\begin{array}{c}
		0 \\ b\end{array}\right).	
\end{align*}
Using \eqref{eq:freschroresolv} and \eqref{eq:dr1}, we have (for $z|x-y|<1$, $0<z\ll 1$, $\lambda=\sqrt{z^2+m^2}$)
\begin{multline} \label{r0temp}  
	\mR_0^\pm(\lambda)=	 
	\left[-i\alpha \cdot \nabla + 2m I_1 +\frac{z^2}{2m} I +\widetilde O(z^4) I\right]    \\                                                             
	\left[g^\pm(z)+G_0 + \widetilde O_2(z^2|x-y|^2\log(z|x-y|))\right].
\end{multline}
We define the function $\log^-(y):=-\log (y) \chi_{\{0<y<1\}}$ and use the following slightly modified lemma from \cite{egd}.
 \begin{lemma}\label{lem:R0exp}	
	We have the following expansion for the kernel of the free resolvent, $\lambda=\sqrt{m^2+z^2},$ $0<z\ll 1$
	\begin{align}
	 \label{eq:R0exp}
		\mR_0^{\pm}(\lambda)(x,y) =2m  g^\pm (z) M_{11}+\mathcal G_0(x,y)
		+E_0^{\pm}(z)(x,y),
	\end{align} 
	where
	\begin{align} \label{def:mG0}
	\mathcal G_0&=-i\alpha \cdot \nabla G_0 
	+2mG_0 I_{1}
\end{align} $E_0^{\pm}$ satisfies the bounds 
	$$
		|E_0^{\pm}|\les z^{k }(|x-y|^{k }+\log^-|x-y|), \,\,
		|\partial_z E_0^{\pm}|\les z^{k-1}(|x-y|^{k }+\log^-|x-y|),  
	$$
	for any $\frac12 \leq k<2$.
	Furthermore, we have
	$$ 
	|\partial^2_z E_0^{\pm}|\les z^{-\f 12}(|x-y|^{\f 32 }+\log^-|x-y|).
	$$
\end{lemma}
We note that the bound on $\partial_z^2 E_0^{\pm}(z)$ is new.  The proof of this bound follows similarly to the proof of Lemma~2.2 in \cite{egd}.  We note that the growth in $|x-y|$ occurs from when derivatives hit the phase in \eqref{largr}, specifically From the  expansion \eqref{largr}, we see $|\partial_z^j\mR_0^\pm (\lambda)(x,y)|\les z^{-\f12}|x-y|^{j-\f12}$ when $z|x-y|\gtrsim 1$.  
\begin{rmk} \label{rmk:rdif}
Note that using \eqref{R0 def} one can obtain (with $r=|x-y|$)
\begin{align*}
	[\mR^{+}_0(\lambda)-\mR^{-}_0(\lambda)](x,y) &= [-i \alpha \cdot \nabla + 2m M_{11} +  		
	\widetilde{O}(z^2)][J_0(z)(x,y)] \\
	& = 2m M_{11} J_0(z) + \widetilde{O}_2( z^2 r+z^2) 
\end{align*}
 for $z|x-y|<1$. For $0<z\ll 1$ we obtain 
 \begin{align}\label{R0dif}
[\mR^{+}_0(\lambda)-\mR^{-}_0(\lambda)](x,y)& = mi M_{11}+ E_1(z)(x,y)
\end{align}
where 
$$
|E_1 (z)(x,y)| \les z^{\f12}  \la r\ra^{\f12}, \,\,\  |\partial_zE_1 (z)(x,y)| \les z^{-\f12}\la r\ra^{\f12}, \,\,\, |\partial^2_zE_1 (z)(x,y)| \les z^{-\f12}\la r\ra^{\f32}.
$$
\end{rmk} 

To obtain expansions for $\mathcal{R}^{\pm}_V(\lambda)=(D_m +V - (\lambda\pm i 0))^{-1} $ where $ \lambda= \sqrt {z^2 + m^2}$  we utilize the symmetric resolvent identity. Since  the matrix 
$V:\mathbb R^2 \to \mathbb C^{2}$ is self-adjoint,
the spectral theorem allows us to write
$$
	V=B^*\left(\begin{array}{cc}
		\lambda_1 & 0 \\ 0 &\lambda_2
	\end{array}\right)B
$$
with $\lambda_j \in \mathbb R$.  We 
further write $\eta_j =|\lambda_j|^{\f12}$,
\begin{align*}
	V=B^*\left(\begin{array}{cc}
	\eta_1 & 0 \\ 0 & \eta_2
	\end{array}\right) U \left(\begin{array}{cc}
	\eta_1 & 0 \\ 0 & \eta_2
	\end{array}\right)B = v^*Uv,
\end{align*}	
	where
\begin{align}\label{eq:vabcd} 
	U=\left(\begin{array}{cc}
	\textrm{sign}(\lambda_1) & 0 \\ 0 & \textrm{sign}(\lambda_2)
	\end{array}
	\right),\,\,\,\text{ and }\,\,
 v=\left(\begin{array}{cc}
		a & b\\ c &d
	\end{array}\right):=\left(\begin{array}{cc}
	\eta_1 & 0 \\ 0 & \eta_2
	\end{array}\right)  B.
\end{align}
Note that the entries of $v$ are $\les \la x\ra^{-\delta/2}$, provided that the entries of $V$ are $\les \la x\ra^{-\delta}$. This representation of $V$ allows us to employ the
symmetric resolvent identity to write the perturbed
resolvent $\mR_V(\lambda)=(D_m+V-\lambda)^{-1}$ as (with $\lambda=\sqrt{m^2+z^2}$, $0<z\ll 1$)
\begin{align}\label{symmresid}
	\mR_V(\lambda)=\mR_0(\lambda)-\mR_0(\lambda)v^* (U+v\mR_0(\lambda)v^*)^{-1}v\mR_0(\lambda).
\end{align}
Our goal  is to invert the operator
\be \label{eq:Mpm}
M^\pm (z)=U+v\mR_0^\pm \left(\sqrt{m^2+z^2}\right)v^*
\ee 
in a  neighborhood of $z=0$.
Recall that $\mR_0(\lambda)(x,y) =2m  g^\pm (z) M_{11}+\mathcal{G}_0(x,y)
		+E_0^{\pm}(z)(x,y)$.
Therefore,
$$
M^\pm (z)=U+v\mathcal{G}_0v^* + 2m  g^\pm (z) vM_{11}v^*+vE_0^{\pm}(z) v^*.
$$
Recalling \eqref{eq:vabcd}, for $f=(f_1,f_2)^T\in L^2 \times L^2$, we have 
 \begin{align*}
	vM_{11}v^*f(x)
	 =\left(\begin{array}{c}
		a(x)  \\ c(x) 
		\end{array}\right)  
		\int_{\R^2}\overline  a(y) f_1(y)+\overline{c}(y)f_2(y)\, dy.
\end{align*}
Thus, we arrive at
\begin{align}\label{vM11v}
	vM_{11}v^*=\|(a,c)\|_2^2 P,
\end{align}
where $P$ is the projection onto the vector $(a,c)^T$.  We also define the
operators $Q:=1-P$, $T:=U+v\mathcal G_0 v^*$, and let
$$
		\mathbbm g^{\pm}(z):=2m \|(a,c)\|_2^2 g^\pm(z).
$$
 We have 
\begin{lemma} \label{lem:M_exp} 
	For $ 0<z\ll 1$, we have 
	\begin{align*}
		M^{\pm}(z)= \mathbbm g^\pm (z) P+ T+M_0^{\pm}(z),
	\end{align*}
where,\footnote{The Hilbert-Schmidt norm of an integral operator
$K$ with integral kernel $K(x,y)$ is defined by
$$
	\|K\|_{HS}^2= \int_{\R^4} |K(x,y)|^2\, dx\, dy.
$$} for any $\frac{1}{2}\leq k<2$,
	\begin{align*}
		 \big\| \sup_{0<z\ll 1} z^{j-k} | \partial_z^j M_0^{\pm}(z)(x,y)|\big\|_{HS}\les 1,\,\,\,\,j=0,1, 
	\end{align*}
	if $|v_{ij}(x)|\lesssim \langle x\rangle^{-\beta}$ for some $\beta>1+k$. 
Moreover,
 \begin{align*}
  \big\| \sup_{0<z\ll 1} z^{{\f 1 2}} | \partial_z^2 M_0^{\pm}(z)(x,y)|\big\|_{HS} \les 1,	
  \end{align*} 
 if $|v_{ij}(x)| \lesssim \langle x\rangle^{- {\f 5 2}} $. 	
\end{lemma}

\begin{proof} Note that by \eqref{eq:Mpm}, Lemma~\ref{lem:R0exp}, and the discussion above, we have
$ M_0=vE_0v^*$. 
Therefore the statement for $j=0,1 $ and for the second derivative follows from the error bounds in Lemma~\ref{lem:R0exp}, and the fact that 
$ (|x-y|^{\ell }+ \log^-|x-y|) \la x\ra^{-\beta} \la y\ra^{-\beta}$ is a Hilbert-Schmidt kernel for $\beta>1+\ell$ and $\ell>-1$. 
\end{proof}

We employ the following terminology from \cite{Sc2,eg2,eg3}:

\begin{defin}
	We say an operator $T:L^2(\R^2) \to   L^2(\R^2)$ with kernel
	$T(\cdot,\cdot)$ is absolutely bounded if the operator with kernel
	$|T(\cdot,\cdot)|$ is bounded from $  L^2(\R^2)$ to $ L^2(\R^2)$. 	
\end{defin}

We note that Hilbert-Schmidt and finite-rank operators are absolutely bounded operators.
As in the case of the Schr\"odinger operator, the invertibility of the leading term of $M^\pm(z)$ depends on the regularity of the threshold energy.    We recall Definition~4.3 in \cite{egd}.

\begin{defin}\label{resondef}
\begin{enumerate}
\item
 Let $Q=1-P$. We  say that $\lambda=m$ is a regular point of the spectrum of
	$H=D_m+V$ provided that $QTQ=Q(U+v\mathcal{G}_0v^*)Q$ is invertible on $Q(L^2\times L^2)$.
	If $QTQ$ is invertible, we denote $D_0:=(QTQ)^{-1}$ as an operator
	on $Q(L^2\times L^2)$. 
\item Assume that $m$ is not a regular point of the spectrum. Let $S_1$ be the Riesz projection
onto the kernel of $QT Q$ as an operator on $Q(L_2 \times L_2)$. Then $QT Q + S_1$ is invertible on $Q(L_2 \times L_2)$.  Accordingly, with a slight abuse of notation we redefine $D_0 = (QT Q+S_1)^{-1} $
as an operator on $Q(L_2 \times L_2)$. We say there is a resonance of the first kind at m if the
operator $T_1 := S_1T P T S_1$ is invertible on $S_1(L_2 \times L_2)$ and define $D_1 := T_1^{-1}$. 
\end{enumerate}
\end{defin}

\begin{rmk} \label{rmk:resonances}  
\begin{enumerate}[(i)]
\item We do our analysis in the positive portion of the spectrum $[m,\infty)$ and develop expansions of $\mathcal{R}_V$ around the threshold $\lambda=m$. One can do a similar analysis for the negative portion of the spectrum taking $\lambda = -\sqrt{z^2+m^2}$. In this case the perturbed equation has a threshold resonance or eigenvalue at $\lambda=-m$ is related to distributional solutions of  $ (H+ mI )g=0$, see Section~7 in \cite{egd} for a detailed characterization.
\item The operator $S_1$ is defined to be the Riesz projection on to the kernel of $QTQ$, see Definition~4.3 and Remark~4.4 in \cite{egd}.  In particular, we have that $S_1\leq Q$, so for any $\phi \in S_1 L^2$, $P\phi=0$, i.e.
$$M_{11}v^*\phi=0.$$
\item The operator $QD_0Q$ is absolutely bounded in $L^2\times L^2$, see Lemma~7.1 in \cite{egd}.
\end{enumerate}
\end{rmk}

The following Lemma is a slight modification of Lemma~4.5 in \cite{egd}.   In particular, we now need control of the second derivative of $E^\pm(z)$.

\begin{lemma}\label{lem:Minverseregular} Assume that $m$ is a regular point of the spectrum of $H$. Also assume that $|v_{ij}(x)|\les \la x\ra^{-\frac52-}$. Then 
$$(M^\pm(z))^{-1}=h^\pm(z)^{-1} S +QD_0Q+E^\pm(z)$$
where 
$$S=\left[\begin{array}{cc}
P& -PTQD_0Q\\ -QD_0QTP &QD_0QTPTQD_0Q
\end{array} \right],$$ 
 $h^\pm(z)=\mathbbm g^\pm(z)+\,$trace$\,(PTP-PTQD_0QTP)$,
 $S$ is a self-adjoint, finite rank operator, and
 \begin{align*}
		\big\| \sup_{0<z\ll 1} z^{j-1/2} | \partial_z^j E^{\pm}(z)(x,y)|\big\|_{HS}\les 1,\,\,\,\,j=0,1, 2
	\end{align*}
 
 \end{lemma}
 \begin{proof} 
	We consider only the `+' case, the `-' proceeds identically.  Let
 \begin{align*} A(z)= \mathbbm g^{+} (z) P+ T =  \left[\begin{array}{cc}
\mathbbm g^{+}(z) P + PTP & PTQ  \\ QTP &QTQ 
\end{array} \right].
 \end{align*}
Then by Feshbach formula (see  Lemma~2.8 in \cite{eg2}) we have  $A^{-1}(z) = h^{+}(z)^{-1} S +QD_0Q$. Hence, using the equality
$$
M^{+} (z) = A(z) + E_0^{+}(z) = (I + E_0^{+}(z) A^{-1}(z)) A(z),
$$
and Neumann series expansion we obtain
$$
(M^{+}(z))^{-1}=A^{-1}(z)(I + E_0^{+}(z) A^{-1}(z))^{-1}= h^{+}(z)^{-1} S +QD_0Q+E^{+}(z).
$$

Note that  as an absolutely bounded operator on $L^2 \times L^2$, $ | \partial_z^j \{A^{-1}(z)\} | \les z^{-j} $ for $j=0,1,2$. Hence, the bounds on $E_0^+ (z)$ in Lemma~~\ref{lem:R0exp} establish the statement.
  \end{proof}

\section{ Nonweighted dispersive estimate} \label{sec:nonweight}

In this section we prove Theorem~\ref{th:main}. We divide this section into three subsections. In the first two subsections we analyze the low energy portion of the Stone's formula, \eqref{eq:Stone}.  First we consider the case when the threshold energies are regular, then we consider the effect of the s-wave resonance(s). To do so, we take a smooth, even cut-off $\chi \in C_c^\infty(\R)$ with $\chi(z)=1$ for $|z|<z_0$ and $\chi(z)=0$ for $|z|>2z_0$.  
\begin{theorem} \label{thm:mainineq}
	Let $|V(x)| \les \la x \ra^{-3-}$. Then, if the threshold $ m$ is regular or if there is resonance of the first kind at $\lambda= m$, then we have
	\be\label{noweight}
		 \int_0^\infty e^{-it\sqrt{z^2+m^2}} \frac{z \chi(z) }{\sqrt{z^2+m^2}} \
		 [\mR_V^+(z)-\mR_V^-(z)](x,y) d\lambda = O( \la t\ra^{-1} ).
      \ee
	\end{theorem}
In Section~\ref{subnonhigh} we prove a high energy result restricted to dyadic energy levels.  In particular, Proposition~\ref{prop:2highgamma} asserts for $j\in \mathbb N$ and $\chi_j(z)$ a smooth cut-off to the interval $z\approx 2^j$,
\begin{align*}
	\sup_{x,y} \bigg|  \int_0^\infty e^{-it\sqrt{z^2+m^2}} \frac{z \chi_j(z) }{\sqrt{z^2+m^2}} \
		 [\mR_V^+(z)-\mR_V^-(z)](x,y) d\lambda \bigg| \les \min(2^{2j}, 2^{7j/2} |t|^{-1}).
\end{align*}
Combing these bounds, and summing over $j$, proves Theorem~\ref{th:main1}.

\subsection{Small energy dispersive estimate in the case that $m$ is regular}

As usual we prove Theorem~\ref{th:main1} considering the Stone's formula. Recalling the symmetric resolvent identity, \eqref{symmresid}, we have
$$ 
	\mathcal{R}_V^\pm(\lambda) = \mR_0^\pm(\lambda) - 
	\mR^\pm_0(\lambda)v^*M^{-1}_{\pm}v\mR^\pm_0(\lambda).
$$
  The contribution of the first term containing only a single free resolvent $ \mathcal{R}_0$ to \eqref{noweight} is controlled by $\la t\ra^{-1}$ in Theorem~3.1 in \cite{egd}. 
To control the contribution of the second term to the Stone's formula, we use the following expansion of the resolvent when $0< z \ll 1$, see (39) in \cite{egd},
$$ \mathcal{R}_0^\pm(\lambda) = \mR_1(z) + \mR_2^{\pm}(z)  + \mR_3^{\pm}(z) +  \mR_4^{\pm}(z) +  \mR_5^{\pm}(z), $$
where
\begin{align*}
& \mR_1(z)(x,y):= -{\f 1 {2\pi}} \chi(z|x-y|) [ -i\alpha \cdot \nabla] \log (z|x-y|)=\frac{i}{2\pi} \chi(z|x-y|) \frac{\alpha \cdot (x-y)}{|x-y|^2}, \\
& \mR_2^{\pm}(z)(x,y):=  \chi(z|x-y|) \big(-i\alpha \cdot \nabla  R_0^{\pm} (z^2)(x,y)   \big)-\mR_1 , \\
&\mR_3 ^{\pm}(z)(x,y) = \chi(z|x-y|) e(z) R_0^{\pm} (z^2)(x,y) , \qquad e(z) = \widetilde{O}_1(z^2), \\
&\mR_4^{\pm}(z)(x,y) =\tilde{ \chi} (z|x-y|) e^{\pm iz |x-y|} \omega_1\big(z(x-y)\big),  \\&  \mR_5^{\pm}(z)(x,y)= 2mI_1 R_0^{\pm} (z^2)(x,y) .
 \end{align*}
 Here $\omega_1$ satisfies the same bounds as $z\omega$ in \eqref{largr}. 
 We refer the reader to the discussion following Theorem~5.1 in \cite{egd} for the required bound for the terms that do not involve $\mR_1$ or $\mR_4^\pm$. These cases boil down to the proof given in \cite{eg2} for the Schr{\oo}dinger operator.  Hence, it is suffices to consider the terms containing $\mR_1$ and $\mR_4^{\pm}$ on the left.
We start with proving Theorem~\ref{thm:mainineq} when there is $\mR_1$ on the left. We write the operator $ \mathcal{R}_0^\pm(\lambda)$ on the right as $ \mathcal{R}_1+  \mathcal{R}^{\pm}_{L,2} +  \mathcal{R}^{\pm}_H $, where $\mR_1$ is as above and
\begin{align*} 
&\mathcal{R}^{\pm}_{L,2}(z)(x,y):= \chi(z|x-y|) \mathcal{R}^\pm_0(\lambda) (x,y) - \mathcal{R}_1(x,y) \\
&  \mathcal{R}^{\pm}_H (z)(x,y) := \widetilde{\chi} (z|x-y|) \mathcal{R}^\pm_0(\lambda) (x,y) = e^{\pm iz r} \tilde{\omega}\big(z(x-y)\big) 
\end{align*}
where $\tilde{\omega}$ satisfies the same bound as $\omega$ in \eqref{largr}.
\begin{rmk} \label{rmk:RL2} Following the proof of Lemma~\ref{lem:R0exp} observe that  (with $r=|x-y|$)
\begin{align} \label{eq:rl2}
\mathcal{R}^{\pm}_{L,2}(z)(x,y)=  \chi(zr) [ 2m g^{\pm}(z) M_{11} +2m G_0I_1] + E^{\pm}_2 (zr),
\end{align}
$|\partial^j_z{E^{\pm}_2} | \les z^{l-j}[ r^{l} + \log ^{-} r] $ for $ j=0,1$ and $0\leq l < 2$.
\end{rmk}

Hence, we need to understand the contribution of the following term to the Stone's formula,
\begin{align}
\label{R1exp}
\mathcal{R}_1v^*M^{-1}_{\pm} v \mathcal{R}_0 = \mathcal{R}_1v^*M^{-1}_{\pm} v \mathcal{R}_1 + \mathcal{R}_1v^*M^{-1}_{\pm} v \mathcal{R}^{\pm}_{L,2} + \mathcal{R}_1v^*M^{-1}_{\pm} v\mathcal{R}^{\pm}_H.
\end{align}

In \cite{egd}, the authors studied the solution operator as an operator $\mathcal H^1\to BMO$ because the operator  $\mR_1$ is not bounded from $L^1 \rightarrow L^2$ or from $L^2 \rightarrow L^{\infty}$. In order to overcome with this hurdle we use the iterated resolvent identity for the operator $M_\pm^{-1}(z) =\big( U + v \mathcal{R}_0^\pm (z) v^*\big)^{-1}$ rather than iterating the Dirac resolvents,  to write 
\begin{align}
 &\label{FIT} M_{\pm}^{-1}(z)= U - Uv \mR^{\pm}_0(z) v^*M_{\pm}^{-1}(z)\,\,\,\text{ and } \\
 & \label{SIT} M_{\pm}^{-1}(z)= U - Uv \mR^{\pm}_0(z)v^*U + Uv\mR^{\pm}_0(z)v^*M_{\pm}^{-1}(z) v \mR^{\pm}_0(z)v^*U .
\end{align} 
These lead us to consider the $L^1 \rightarrow L^2$ norm (or $L^2 \rightarrow L^{\infty}$ norm) of the operator $ \mR_1V\mR_1$ rather than the operator $\mR_1$.

Using \eqref{SIT} we have 
\be \label{eq:R1MR1}
\mR_1v^*M^{-1}_{\pm} v\mR_1= \mR_1V\mR_1 - \mR_1V\mR_0^{\pm}V\mR_1 + \mR_1V\mR_0^{\pm}v^{*}M^{-1}_{\pm} v\mR_0^{\pm}V\mR_1.
\ee
We note that the first term would not be bounded uniformly in $x,y$ due to the singular behavior of $\mR_1$.  However, since we take the difference of the `+' and `-' terms in the Stone's formula, \eqref{eq:Stone}, these terms cancel each other.
Hence, we consider
\begin{multline} \label{mr1}
\mR_1v^*M^{-1}_{+} v\mR_1-\mR_1v^*M^{-1}_{-} v\mathcal{R}_1 = -\mR_1V[\mR_0^{+}-\mR_0^{-}]V\mR_1 \\+\mR_1V[\mR_0^{+}v^{*}M^{-1}_{+}v\mR_0^{+} - \mR_0^{-}v^{*}M^{-1}_{-}v\mR_0^{-}]V\mR_1.
\end{multline}
Using the expansion for $(M^\pm(z))^{-1}$ in  Lemma~\ref{lem:Minverseregular} we write
$$
\mR_1v^*M^{-1}_{+} v\mR_1-\mR_1v^*M^{-1}_{-} v\mathcal{R}_1=-\Gamma^1_1+\Gamma^1_2+\Gamma^1_3+\Gamma^1_4,
$$
where
\begin{align}\label{eq:Gamma11}
\Gamma^1_1&:=\mR_1V[\mR_0^{+}-\mR_0^{-}]V\mR_1,\\
\label{eq:Gamma12}	\Gamma^1_2& := \mR_1V \big(\mR_0^{+ } v^{*} \frac{S}{h^{+}}v\mR_0^{+}-
	\mR_0^{-} v^{*} \frac{S}{h^{-}}v\mR_0^{-}\big)V\mR_1, \\
	\Gamma^1_3&:=   \mR_1V\big( \mR_0^{+} v^{*}QD_0Qv\mR_0^{+}-\mR_0^{-} v^{*}QD_0Qv\mR_0^{-}\big) V \mR_1,\label{eq:Gamma13}\\
\Gamma^1_4 &:=	 \mR_1V \mR_0^{\pm} v^{*}E^{\pm} v\mR_0^{\pm} V \mR_1. \label{eq:Gamma14}
\end{align}
For the second term on the right hand side of \eqref{R1exp}, we use \eqref{FIT} to obtain
$$
	\mathcal{R}_1v^*M^{-1}_{\pm} v \mathcal{R}^{\pm}_{L,2}= \mathcal{R}_1 V\mathcal{R}^{\pm}_{L,2} - \mathcal{R}_1 V \mR^{\pm}_0  v^*M_{\pm}^{-1} v \mathcal{R}^{\pm}_{L,2}.
$$
Using the expansion for $M_\pm^{-1}(z)$ in  Lemma~\ref{lem:Minverseregular} we write
$$
 \mathcal{R}_1v^*M^{-1}_{+} v \mathcal{R}^{+}_{L,2}- \mathcal{R}_1v^*M^{-1}_{-} v \mathcal{R}^{-}_{L,2}=\Gamma^2_1-\Gamma^2_2-\Gamma^2_3-\Gamma^2_4,
$$
where
\begin{align}\label{eq:Gamma21}
\Gamma^2_1&:= \mathcal{R}_1 V\big(\mathcal{R}^{+}_{L,2}-\mathcal{R}^{+}_{L,2}\big),\\
	\label{eq:Gamma22}\Gamma^2_2& := \mR_1V \big(\mR_0^{+ } v^{*} \frac{S}{h^{+}}v\mathcal{R}^{+}_{L,2}-
	\mR_0^{-} v^{*} \frac{S}{h^{-}}v\mathcal{R}^{-}_{L,2}\big) , \\
	\Gamma^2_3&:=   \mR_1V\big( \mR_0^{+} v^{*}QD_0Qv\mathcal{R}^{+}_{L,2}-\mR_0^{-} v^{*}QD_0Qv\mathcal{R}^{-}_{L,2}\big) ,\label{eq:Gamma23}\\
\Gamma^2_4&:=	 \mR_1V \mR_0^{\pm} v^{*}E^{\pm} v\mathcal{R}^{\pm}_{L,2}. \label{eq:Gamma24}
\end{align}
We will consider the third summand on the right hand side of \eqref{R1exp} later.
 
For a given $z$ dependent operator $\Gamma$ we define 
$$
	I(\Gamma )(x,y):= \int_0^\infty e^{-it\sqrt{z^2+m^2}} \frac{z\chi(z)}{\sqrt{z^2+m^2}} \Gamma(z)(x,y) dz.
$$
Then by integration by parts
\begin{align} \label{eq:ibp}
	I(\Gamma ) = \frac{ie^{-itm}}t  \Gamma |_{z=0} -\frac{i}t\int_0^\infty e^{-it\sqrt{z^2+m^2}}    \partial_z \big[\chi(z) \Gamma(z) \big] dz  
\end{align}
where $\Gamma|_{z=0}$ means $\lim_{z\to 0+} \Gamma(z)$.  This implies that (provided the integral converges)
\begin{align} \label{eq:ibp1}
|I(\Gamma)|\les \frac1t \int_0^\infty \big|\partial_z \big[\chi(z) \Gamma(z) \big]\big| dz,
\end{align}
since the integrability of $\big|\partial_z \big[\chi(z) \Gamma(z) \big]\big|$ implies the boundedness of $\chi(z) \Gamma(z)$ (noting that $\chi(1)=0$). This also implies $ |I(\Gamma)| \les 1$ uniformly in $x$ and $y$. 

Therefore, the bound $\sup_{x,y}|I(\Gamma)|\les \la t\ra^{-1}$ follows from 
\be\label{eq:L1partial}
\sup_{x,y} \big\| \partial_z \big[\chi(z) \Gamma(z) \big]\big\|_{L^1_z}\les 1.
\ee
Below, by showing that \eqref{eq:L1partial} holds, we prove that
$I(\Gamma_j^i)=O(\la t\ra^{-1})$ for each $i=1,2$ and $j=1,2,3,4$.

\begin{lemma} \label{lem:gamma j=1}
	Under the assumptions of Theorem~\ref{th:main1}   we have  $ I(\Gamma^i_1) =O( \la t\ra^{-1})$ uniformly in $x$ and $y$  for each $i=1,2$.
\end{lemma} 
\begin{proof}
Recall by Remark~\ref{rmk:rdif} that  $ [\mR_0^{+}-\mR_0^{-}](z) (x_1,y_1) = im M_{11}+ E_1(z)(x_1,y_1) $, which implies that $ |\partial_z\{ [\mR_0^{+}-\mR_0^{-}] (x_1,y_1) \}| \les z^{-1/2} \la x_1-y_1\ra ^{1/2} $ . Also writing 
\begin{multline}\label{eq:R1 useful}
	\mR_1(z)(x,x_1) = \frac{i}{2\pi} \frac{\alpha \cdot (x-x_1)}{|x-x_1|^2}-\widetilde \chi(z|x-x_1|) \frac{i}{2\pi} \frac{\alpha \cdot (x-x_1)}{|x-x_1|^2} 
	\\ =\frac{i}{2\pi} \frac{\alpha \cdot (x-x_1)}{|x-x_1|^2}+\widetilde O_1(z)
\end{multline}
we conclude that 
\be \label{eq:R1 useful1}
|\mR_1(z)(x,x_1)|\les 1+|x-x_1|^{-1},\,\,\,\,\,\, |\partial_z \mR_1(z)(x,x_1)|\les 1.
\ee
Hence
\begin{multline*}
|\partial_z \{\chi(z) \Gamma_1^1(z)(x,y)\}|  
    \les \\ \int_{\R^4} z^{-1/2}\max \{1, |x-x_1|^{-1}\} |V(x_1)| \la x_1 -y_1 \ra^{1/2} 
    |V(y_1)| \max \{1, |y-y_1|^{-1}\} \  dx_1 dy_1.
\end{multline*}
We have $\| \max \{1, |x-x_1|^{-1}\} V(x_1)\la x_1 \ra \|_{L^1_{x_1}} \les 1$, uniformly in $x$, and since $z^{-1/2}$ is integrable in a neighborhood of zero,  we see that $ \partial_z \big[\chi(z) \Gamma_1^1(z)\big] $ is in $L^1$ uniformly in $x$ and $y$. This yields \eqref{eq:L1partial} and hence finishes the proof for $\Gamma_1^1$.
 
Now we consider $I(\Gamma_2^1)$. Using  Remark~\ref{rmk:rdif}, with  $r_1=|y-x_1|$,  we have 
\be\label{eq:RL2pm}
	[\mathcal{R}^{+}_{L,2} - \mathcal{R}^{-}_{L,2}](x_1,y) 
	= \chi(zr_1)[\mR_0^{+} - \mR_0^{-}](x_1,y)= \chi(zr_1)\big[imM_{11}+\widetilde O_1\big( z^2  r_1+z^2   \big)\big],
\ee 
which implies that   on the support of $\chi(z)$ 
\be\label{eq:pRL2pm}
|\partial_z \{[\mathcal{R}^{+}_{L,2} - \mathcal{R}^{-}_{L,2}](x_1,y)\}| \les \chi (zr_1) [zr_1 + z ] \les 1.
\ee
Also recalling \eqref{eq:R1 useful}, we obtain 
$$
\big\| \partial_z \big[\chi(z)  \Gamma_2^1(z)(x,y)\big]\big\|_{L^1_z}    
  \les  \int_{\R^2} \max\{1, |x-x_1|^{-1}\} |V(x_1)| dx_1 \les 1
$$
uniformly in $x$ and $y$.
This finishes the proof of the lemma.
\end{proof}
 
\begin{lemma} \label{lem:gamma j=2}  Under the assumptions of Theorem~\ref{th:main1}   we have  $ I(\Gamma^i_2) =O( \la t\ra^{-1})$ uniformly in $x$ and $y$  for each $i=1,2$.
 \end{lemma}
\begin{proof} We start with  $\Gamma^2_2$:
$$
\Gamma^2_2= \mR_1V \big(\frac{\mR_0^{+ }}{h^{+}} -\frac{\mR_0^{- }}{h^{-}}\big)  v^{*} S v\mathcal{R}^{+}_{L,2} +\mR_1V  \frac{\mR_0^{- }}{h^{-}}  v^{*} S v\big(\mathcal{R}^{+}_{L,2}-\mathcal{R}^{-}_{L,2}\big).
$$
Using Lemma~\ref{lem:R0exp} with $k=\f12$  for $\mR_0^{\pm}(z)(x_1,x_2)$, and recalling the relationship between $h^\pm(z)$ and $g^{\pm}(z)$ in Lemma~\ref{lem:Minverseregular} , we obtain
$$
	\Big|\frac{\mR_0^{+ }}{h^{+}} -\frac{\mR_0^{- }}{h^{-}}\Big|\les \frac{|x_1-x_2|^{1/2}+|x_1-x_2|^{-1}}{\log^2(z)},
$$
$$
	\Big|\partial_z\Big(\frac{\mR_0^{+ }}{h^{+}} -\frac{\mR_0^{- }}{h^{-}}\Big)\Big|\les \frac{|x_1-x_2|^{1/2}+|x_1-x_2|^{-1}}{z|\log(z)|^3}, 
$$
and
$$
\Big| \frac{\mR_0^{- }}{h^{-}}\Big|\les \frac{|x_1-x_2|^{1/2}+|x_1-x_2|^{-1}}{|\log(z)| },\,\,\,\,\,\Big|\partial_z \Big( \frac{\mR_0^{- }}{h^{-}} \Big) \Big|\les \frac{|x_1-x_2|^{1/2}+|x_1-x_2|^{-1}}{z\log^2(z) }.
$$
Using \eqref{eq:RL2pm} and \eqref{eq:pRL2pm} we have
\be\label{eq:RLdifbounds}
\big|\mathcal{R}^{+}_{L,2}-\mathcal{R}^{-}_{L,2}\big| + \big|\partial_z(\mathcal{R}^{+}_{L,2}-\mathcal{R}^{-}_{L,2})\big|\les 1.
\ee
Recalling Remark~\ref{rmk:RL2} we have
\be\label{eq:RL2bounds}\left\{\begin{array}{l}
\big|\mathcal{R}^{+}_{L,2} \big|\les |\log(z)|+\log^-(|y_1-y|), \\
\big|\partial_z\mathcal{R}^{+}_{L,2} \big|\les \f1z (1+\log^-(|y_1-y|))+|\chi^\prime(z|y-y_1|)|\f{|\log(z)|}{z}.
\end{array}\right.
\ee
For the second bound above observe that on the support of $\chi(z)\chi(zr)$ we have 
$|\log(r)|\les |\log(z)|+\log^-(r)$.

Using these bounds and \eqref{eq:R1 useful1} for $\mR_1$, we obtain (with $r_0=|x-x_1|$, 
 $r_1=|x_1-x_2|$,  $r_2=|y_1-y|$)
\begin{multline*}
\big\|\partial_z\big[\chi(z)\Gamma_2^2(z)\big]\big\|_{L^1_z}\\  \les
 \int_{\R^6} \int_0^{z_0}\frac{(1+r_0^{-1})(r_1^{1/2}+r_1^{-1}) |S(x_2,y_1)| \big(\f{ 1+\log^-(r_2)}{z\log^2(z)} + \f{|\chi^\prime(zr_2)|}{z}\big)}{\la x_1\ra^{3}\la x_2\ra^{\f32}  \la y_1\ra^{\f32}}  dz dx_1dx_2dy_1\\
\les  \int_{\R^6} \frac{(1+r_0^{-1})(r_1^{1/2}+r_1^{-1}) |S(x_2,y_1)| ( 1+\log^-(r_2) )}{\la x_1\ra^{3}\la x_2\ra^{\f32}  \la y_1\ra^{\f32}}  dx_1dx_2dy_1.
\end{multline*}
One can see that this is bounded in $x$ and $y$ using Lemma~\ref{lem:potential} and the absolute boundedness of $S$. This finishes the proof for $\Gamma_2^2$.

Similarly, we write 
$$
\Gamma^1_2= \mR_1V \big(\frac{\mR_0^{+ }}{h^{+}} -\frac{\mR_0^{- }}{h^{-}}\big)  v^{*} S v \mR_0^{+ } V\mR_1 +\mR_1V  \frac{\mR_0^{- }}{h^{-}}  v^{*} S v\big( \mR_0^{+ }- \mR_0^{- }\big) V\mR_1.
$$
Using Remark~\ref{rmk:rdif} we have
\be\label{eq:R0difbounds}\left\{\begin{array}{l}
\big| \mR_0^{+ }- \mR_0^{- }\big|\les \la y_1-y_2\ra^{1/2},\\ \big| \partial_z\big(\mR_0^{+ }- \mR_0^{- }\big)\big|\les z^{-1/2} \la y_1-y_2\ra^{1/2}. 
\end{array}\right.
\ee
Also using Lemma~\ref{lem:R0exp} with $k=\f12$,	we have
\be\label{eq:R0bounds}\left\{\begin{array}{l}
\big| \mR_0^{+ } \big|\les |\log(z)| \big(\la y_1-y_2\ra^{1/2}+|y_1-y_2|^{-1}\big),\\ \big| \partial_z \mR_0^{+ } \big|\les \f1z \big(\la y_1-y_2\ra^{1/2}+|y_1-y_2|^{-1}\big).
\end{array}\right.
\ee
We conclude that  (with $r_0=|x-x_1|$, 
 $r_1=|x_1-x_2|$,  $r_2=|y_1-y_2|$, $r_3=|y_2-y|$)
\begin{multline*}
\big\|\partial_z\big[\chi(z)\Gamma_2^1(z)\big]\big\|_{L^1_z}\\  \les
 \int_{\R^8} \int_0^{z_0}\frac{(1+r_0^{-1})(r_1^{1/2}+r_1^{-1}) |S(x_2,y_2)|(r_2^{1/2}+r_2^{-1})   (1+r_3^{-1})}{\la x_1\ra^{3}\la x_2\ra^{\f32}   \la y_2\ra^{\f32} \la y_1\ra^{3}z\log^2(z)}  dz dx_1dx_2dy_1 dy_2\\
\les  \int_{\R^8}\frac{(1+r_0^{-1})(r_1^{1/2}+r_1^{-1}) |S(x_2,y_2)|(r_2^{1/2}+r_2^{-1})   (1+r_3^{-1})}{\la x_1\ra^{3}\la x_2\ra^{\f32}   \la y_2\ra^{\f32} \la y_1\ra^{3} }   dx_1dx_2dy_1 dy_2.
\end{multline*}
This finishes the proof for $\Gamma^1_2$ using Lemma~\ref{lem:potential} as above.
\end{proof}
\begin{lemma}\label{lem:gamma j=3} Under the assumptions of Theorem~\ref{th:main1}   we have  $ I(\Gamma^i_3) =O( \la t\ra^{-1})$ uniformly in $x$ and $y$  for each $i=1,2$.
\end{lemma}
\begin{proof}
We will give the proof only for $\Gamma^2_3$; the proof for $\Gamma^1_3$ is similar but easier. We rewrite
$$
\Gamma^2_3= \mR_1V \big( \mR_0^{+ }   - \mR_0^{- } \big)  v^{*} QD_0Q v\mathcal{R}^{+}_{L,2} +\mR_1V   \mR_0^{- }   v^{*} QD_0Q  v\big(\mathcal{R}^{+}_{L,2}-\mathcal{R}^{-}_{L,2}\big).
$$
Recall that $M_{11}v^*Q=QvM_{11}=0$ by definition of the projection $Q$. Recalling \eqref{R0dif} we can replace $\mR_0^{+ }(z)(x_1,x_2)   - \mR_0^{- }(z)(x_1,x_2) $ in the first summand with 
$$
\mR_0^{+ }   - \mR_0^{- }-imM_{11}=E_1(z)=\widetilde O_1(z^{1/2}\la x_1-x_2\ra^{1/2}).
$$
Similarly, in the second summand we replace $\mathcal{R}^{+}_{L,2}(z)(y_1,y)- \mathcal{R}^{-}_{L,2}(z)(y_1,y)$ with
$$
\mathcal{R}^{+}_{L,2} - \mathcal{R}^{-}_{L,2} -im M_{11} \chi(z\la y\ra) = imM_{11}(\chi(z|y-y_1|)-\chi(z\la y\ra))+\widetilde O_1 ( z  )=\widetilde O_1 \big( (z\la y_1\ra)^{0+}\big).
$$
In the first equality we used  \eqref{eq:RL2pm}, and the second equality follows from the mean value theorem.
 
 Combining these bounds with \eqref{eq:R1 useful1}, \eqref{eq:RL2bounds}, and \eqref{eq:R0bounds} we obtain
  (with $r_0=|x-x_1|$, 
 $r_1=|x_1-x_2|$,  $r_2=|y_1-y|$)
\begin{multline}\label{eq:j=3proof}
\big\|\partial_z\big[\chi(z)\Gamma_3^2(z)\big]\big\|_{L^1_z}\\  \les
 \int_{\R^6} \int_0^{z_0}\frac{(1+r_0^{-1})(r_1^{1/2}+r_1^{-1}) |QD_0Q(x_2,y_1)|  (1+\log^-(r_2) \big)}{\la x_1\ra^{3}\la x_2\ra^{\f32}  \la y_1\ra^{\f32-} z^{1-}}  dz dx_1dx_2dy_1.
\end{multline}
One can see that this is bounded in $x$ and $y$ using the integrability of $z^{-1+}$ on $[0,z_0]$, Lemma~\ref{lem:potential}, and the absolute boundedness of $QD_0Q$. This finishes the proof.
\end{proof}

\begin{lemma} \label{lem:gamma j=4}  Under the assumptions of Theorem~\ref{th:main1}   we have  $ I(\Gamma^i_4) =O( \la t\ra^{-1})$ uniformly in $x$ and $y$  for each $i=1,2$.
\end{lemma} 
\begin{proof} The proof is similar to the proof of previous three lemmas. Instead of cancellation between $\pm$ terms or orthogonality,  one uses the smallness of the error term in $z$, see 
Lemma~\ref{lem:Minverseregular}. We omit the details. 
 \end{proof}

To estimate the third term, $\mathcal{R}_1v^*M^{-1}_{\pm} v\mathcal{R}^{\pm}_H$, in \eqref{R1exp} we use Lemma~3.4 from \cite{egd}. 
\begin{lemma}\label{stat phase}

	Let $\phi_ \pm(z):=\sqrt{z^2+m^2} \pm \frac{z r}{t}$, if 
	$$
	|a(z)|\les 	\frac{z\chi(z) \widetilde \chi(zr)}{(1+zr)^{\f12}} ,
	\qquad |\partial_z a(z)|\les 
	\frac{\chi(z) \widetilde \chi(zr)}{(1+zr)^{\f12}}
	$$	
	then we have the bound
	$$
    		\bigg| \int_{0}^{\infty} e^{-it\phi_{\pm}(z)} a(z)\, dz \bigg|
    		\lesssim \la t\ra^{-1}.
	$$

\end{lemma}


\begin{lemma} \label{lem:r1rh}We have $ |I(\mathcal{R}_1v^*M^{-1}_{\pm} v\mathcal{R}^{\pm}_H)| \les \la t \ra ^{-1} $. 
\end{lemma}
\begin{proof} Using \eqref{FIT} we have 
$$\mathcal{R}_1v^*M^{-1}_{\pm} v\mathcal{R}^{\pm}_H = \mathcal{R}_1V\mathcal{R}^{\pm}_H+ \mathcal{R}_1V \mR_0^{\pm} v^*M^{-1}_{\pm} v\mathcal{R}^{\pm}_H. $$
We start estimating the first term. Note that we have 
$$
\int_0^\infty e^{-it\sqrt{z^2+m^2}} \frac{z\chi(z)}{\sqrt{z^2+m^2}}  [\mathcal{R}_1V\mathcal{R}^{\pm}_H ] (x,y)  dz =   \int_{\R^2} \int_{0}^{\infty}  e^{-it\phi_{\pm}(z)} a(z)  \,  dz  dy_1
$$
where $r=|y-y_1|$ and 
$$a(z) =  \frac{z\chi(z) \widetilde \chi(zr)}{\sqrt{z^2+m^2}} [\mathcal{R}_1V](x,y_1) \tilde{\omega}_{\pm}(zr). $$
Here $\tilde{\omega}$ satisfies the same bound as $\omega$ in \eqref{largr}.
By using \eqref{eq:R1 useful}, we may immediately use Lemma~\ref{stat phase} and integrate in $y_1$ since $(1+|x-y_1|^{-1})V(y_1)\in L^1_{y_1}$ uniformly in $x$.

	The second term is bounded similarly.  
	Recall the expansion for $M^{-1}_{\pm}$ from Lemma~\ref{lem:Minverseregular} and the expansion in Lemma~\ref{lem:R0exp} for $R_0^{\pm}$ and the definition of $\mR_H$.  To apply Lemma~\ref{stat phase} to obtain the desired time decay, we need only show that $\mR_1V\mR_0^\pm v^* M_{\pm}^{-1}v=\widetilde O_1(1)$ in the spectral variable, and converges in an appropriate sense.  The convergence of the spatial integrals has been established in Lemma~\ref{lem:gamma j=2} for example.  The most simple estimate would yield that $\mR_1V\mR_0^\pm v^* M_{\pm}^{-1}v=\widetilde O_1(\log z)$.  This bound is not sharp as the $\log z$ behavior arises from when the most singular terms in $\mR_0^\pm(z)$ and the $M_{\pm}^{-1}(z)$ interact.  However, using the expansions in Lemma~\ref{lem:R0exp} and \ref{lem:Minverseregular}, the most singular terms are
	\begin{align*}
		\mR_1V\mR_0^\pm v^* M_{\pm}^{-1}v=\mR_1 V [ 2m g^\pm(z) M_{11} ]v^* QD_0Q v	+\widetilde O_1(1).
	\end{align*}
	Using the orthogonality $ M_{11} Q v^* =0$ the first term vanishes and we have the needed bounds to apply Lemma~\ref{stat phase}. 
\end{proof}

Lastly we consider the contribution of $\mR_4^\pm v^*M_{\pm}^{-1}v \mR_0^{\pm}$ to the Stone's formula, \eqref{noweight}. As before we write
\be \label{eq:R4R0}
\mR^{\pm}_4v^*M_{\pm}^{-1}v \mR_0^{\pm} = \mR^{\pm}_4v^*M_{\pm}^{-1}v \mR_1+ \mR^{\pm}_4v^*M_{\pm}^{-1}v \mR_{L,2}^\pm + \mR^{\pm}_4v^*M_{\pm}^{-1}v \mR_H^\pm .
\ee
The proof for the first two terms is similar to the one in Lemma~\ref{lem:r1rh} above involving $\mR_H$. It is in fact easier since $\mR_4$ is comparable to  $z \mR_H$. For the last term we refer the reader to the portion of the proof of  Proposition~5.3 in \cite{egd} concerning the operator $\Gamma_3$.  The statement of this proposition asserts a bound from $H^1$ to $ BMO$, however the  argument yields  an $L^1\to L^\infty$ bound.

\subsection{Small energy dispersive estimates in the case of an s-wave resonance}

We need to consider the following terms (see the expansion given by Lemma~4.6 in \cite{egd}): 
\begin{align}\label{eq:Gamma15}
\Gamma^1_5&:=  \mR_1V \big(\mR_0^{+ } v^{*}   h^{+} S_1D_1S_1 v\mR_0^{+}-
	\mR_0^{-} v^{*}  h^{-} S_1D_1S_1v\mR_0^{-}\big)V\mR_1, \\
	\Gamma^1_6&:=   \mR_1V\big( \mR_0^{+} v^{*}Av\mR_0^{+}-\mR_0^{-} v^{*}Av\mR_0^{-}\big) V \mR_1,\label{eq:Gamma16} \\
	\label{eq:Gamma25}\Gamma^2_5& := \mR_1V \big(\mR_0^{+ } v^{*} h^{+} S_1D_1S_1 v\mathcal{R}^{+}_{L,2}-
	\mR_0^{-} v^{*}h^{-} S_1D_1S_1 v\mathcal{R}^{-}_{L,2}\big) , \\
	\Gamma^2_6&:=   \mR_1V\big( \mR_0^{+} v^{*}Av\mathcal{R}^{+}_{L,2}-\mR_0^{-} v^{*}Av\mathcal{R}^{-}_{L,2}\big).\label{eq:Gamma26} 
\end{align}
Here $A=SS_1D_1S_1+S_1D_1S_1S$, which  is an absolutely bounded finite rank operator with no $z$ dependence. However, unlike $QD_0Q$, the orthogonality property holds only on one side.  The other terms in the expansion are similar to the ones we discussed in the regular case and are controlled by Lemmas~\ref{lem:gamma j=1}--\ref{lem:gamma j=4}.

\begin{lemma} \label{lem:gamma j=5}
	Under the assumption of Theorem~\ref{th:main1} and for each $i=1,2$ we have  $ I(\Gamma^i_5) =O( \la t\ra^{-1})$ uniformly in $x$ and $y$.
\end{lemma} 
\begin{proof}
We only discuss $\Gamma^2_5$; the proof for $\Gamma^1_5$ is similar.
We need to consider the following operators:
\begin{align*}
	\Gamma^2_{5,1}&:=\mR_1V \big(\mR_0^{+ }-\mR_0^{-}\big) v^{*} h^{+} S_1D_1S_1 v\mathcal{R}^{+}_{L,2},\\
	\Gamma^2_{5,2}&:=\mR_1V  \mR_0^{-} v^{*} \big(h^{+}-h^-\big) S_1D_1S_1 v\mathcal{R}^{+}_{L,2},\\
	\Gamma^2_{5,3}&:=\mR_0^{-} v^{*}h^{-} S_1D_1S_1 v\big(\mathcal{R}^{+}_{L,2}-\mathcal{R}^{-}_{L,2}\big).
\end{align*}
Since $S_1\leq Q$, the proof of Lemma~\ref{lem:gamma j=3} implies the required bounds for $\Gamma^2_{5,1}$ and $\Gamma^2_{5,3}$ above. In particular, the bound \eqref{eq:j=3proof} remains valid even with the additional factor of $h^\pm(z)$, as the polynomial gain in $z$ obtained in the proof suffices to control the logarithmic behavior of $h^\pm(z)$.

For $\Gamma^2_{5,2}$ observe that
$h^+-h^-$ is a constant. We utilize the orthogonality property $M_{11}v^*Q=QvM_{11}=0$ to replace $R_0^-$ with $\mathcal G_0 +E_0^{-}$, where
(see \eqref{eq:R0exp})
$$|\mathcal G_0 +E_0^{-}|\les |x_1-x_2|^{-1}+|x_1-x_2|^{1/2},\,\,\,\,
|\partial_z(\mathcal G_0 +E_0^{-})|\les z^{-1/2}(|x_1-x_2|^{-1}+|x_1-x_2|^{1/2}).$$ 
 Similarly we replace $\mR^{+}_{L,2}$ with (see Remark~\ref{rmk:RL2})
\begin{multline*}
F(z,y,y_1):=\mR^{+}_{L,2}+\frac{mI_1}{\pi}\chi(z\la y\ra)\log(z\la y\ra) \\ = \frac{mI_1}{\pi}\big(\chi(z\la y\ra)\log(z\la y\ra)-\chi(z|y-y_1|)\log(z |y-y_1|)\big)+E_2^+(z|y-y_1|).
\end{multline*}
Using Remark~\ref{rmk:RL2} for the error term and \cite{Sc2} or \cite[Lemma 3.3]{eg2} for the first term, we have
$$
\sup_{0<z<z_0}|F(z,y,y_1)|+\int_0^{z_0} |\partial_z F(z,y,y_1)| dz \les 1+\log(\la y_1\ra)+\log^-(|y-y_1|). 
$$
To see this inequality for $\partial_z E^+_2$ take $l=0+$ in Remark~\ref{rmk:RL2} and use the support condition.

Using these bounds and \eqref{eq:R1 useful1} for $\mR_1$, we obtain (with $r_0=|x-x_1|$, 
 $r_1=|x_1-x_2|$,  $r_2=|y_1-y|$) 
\begin{multline*}
\big\|\partial_z\big[\chi(z)\Gamma^2_{5,2}(z)\big]\big\|_{L^1_z} \les \\ 
 \int_{\R^6} \int_0^{z_0}\frac{(1+r_0^{-1})(r_1^{1/2}+r_1^{-1}) |S_1D_1S_1(x_2,y_1)| \big( z^{-1/2}  |F | +|\partial_z F |\big)}{\la x_1\ra^{3}\la x_2\ra^{\f32}  \la y_1\ra^{\f32}}  dz dx_1dx_2dy_1
\les  \\ \int_{\R^6} \frac{(1+r_0^{-1})(r_1^{1/2}+r_1^{-1}) |S_1D_1S_1(x_2,y_1)| ( 1+\log(\la y_1\ra)+\log^-(r_2) )}{\la x_1\ra^{3}\la x_2\ra^{\f32}  \la y_1\ra^{\f32}}  dx_1dx_2dy_1.
\end{multline*}
One can see that this is bounded in $x$ and $y$ using Lemma~\ref{lem:potential} and the absolute boundedness of $S_1D_1S_1$. 
\end{proof}
\begin{lemma} \label{lem:gamma j=6}
	Under the assumption of Theorem~\ref{th:main1} and for each $i=1,2$ we have  $ I(\Gamma^i_6) =O( \la t\ra^{-1})$ uniformly in $x$ and $y$.
\end{lemma} 
\begin{proof} We only discuss $\Gamma^2_6$; the proof for $\Gamma^1_6$ is similar.
We rewrite
$$
\Gamma^2_{6 }:=\mR_1V \big(\mR_0^{+ }-\mR_0^{-}\big) v^{*}   A v\mR^{+}_{L,2}+\mR_1V  \mR_0^{-} v^{*} A v\big( \mathcal{R}^{+}_{L,2}-\mathcal{R}^{-}_{L,2}\big).
$$ 
We note that we can use the cancellation only on one side. When we can only use the cancellation on the left, we replace $\mR_0^{+ }-\mR_0^{-}$ with $\mR_0^{+ }  - \mR_0^{- }-imM_{11}$ as in the proof of Lemma~\ref{lem:gamma j=3}, and replace $\mR_0^-$ with $\mathcal G_0 +E_0^{-}$ as in the proof of Lemma~\ref{lem:gamma j=5} for the first and second summands respectively. If the cancellation is on the right,  we replace $\mR^{+}_{L,2}$ with $F$ as in the proof of Lemma~\ref{lem:gamma j=5}, and replace $\mathcal{R}^{+}_{L,2}-\mathcal{R}^{-}_{L,2}$ with $\mathcal{R}^{+}_{L,2} - \mathcal{R}^{-}_{L,2} -im M_{11} \chi(z\la y\ra)$ as in the proof of Lemma~\ref{lem:gamma j=3} for the first and second summands respectively. We leave the details to the interested reader.
\end{proof}

For the remaining terms involving $\mR_4$ or $\mR_H$ see the previous section and the proof of Proposition 5.6 in \cite{egd}.

\subsection{Large energy dispersive estimates} \label{subnonhigh}  
To prove the large energy dispersive bound uniformly in $x$ and $y$, we restrict to dyadic energy levels.  In particular,
we fix $j\in\mathbb N$, and let $\chi_j(z)$ be  a cut-off to
$z\approx 2^j$, and analyze the contribution of the operators $\chi_j(z) [\mR_V^+-\mR_V^-](z)$ to the Stone's formula.

We begin by employing the resolvent expansion
\be \label{exp:resolvent}
	\mathcal R_V^{\pm}(\lambda) =  \mathcal R_0^{\pm}(\lambda)- \mathcal R_0^{\pm}(\lambda) V\mathcal R_0^{\pm}(\lambda)  + \mathcal R_0^{\pm}(\lambda)  V\mR_V^{\pm} (\lambda) V    \mR_0^{\pm}(\lambda).
\ee
We note that the first two terms are bounded by
$ \min(2^{2j}, 2^{5j/2} |t|^{-1}),$
see Lemmas~6.3 and 6.4 of \cite{egd} respectively.  For the  final term we have  
\begin{prop}\label{prop:2highgamma}
	Under the assumptions of Theorem~\ref{th:main1}
	the following bound holds 
	\begin{multline} \label{eqn: 2j hightail}
	\sup_{x,y\in \R^2}
	\bigg|\int_0^\infty e^{-it\sqrt{z^2+m^2}}\frac{z\chi_j(z)}{\sqrt{z^2+m^2}}  \mathcal R_0^{\pm}V\mathcal R_V^{\pm}   V   \mathcal R_0^{\pm}(\lambda)  (x,y)
	\, dz\bigg|  \\  \les  	  \min(2^{2j},   2^{7j/2} |t|^{-1}).
	\end{multline}
\end{prop}

For the  outer resolvents, we will write $\mR_0=\mR_L+\mR_H$ where
\begin{multline}  \label{eqRL}
\mR_L^\pm(z)(x,y)  =  \chi(z|x-y|)\left(\frac{i\alpha\cdot(x-y)}{2\pi|x-y|^2}+\widetilde O_1(z (z |x-y|)^{0-})\right)\\ 
=\chi(z|x-y|)\widetilde O_1(|x-y|^{-1}),
\end{multline}
\begin{multline} \label{eqRH} \mR_H^\pm(z)(x,y) = e^{\pm i z|x-y|}\widetilde w_\pm(z|x-y|), \\  |\partial^k_z[\widetilde w_\pm(z|x-y|)]|    \les z^{1-k} (1+z|x-y|)^{-1/2}.
\end{multline}
The proposition follows from\footnote{Lemma 6.4 in \cite{egd} asserts a bound in the $H^1
 \to BMO$ setting, however the proof yields an $L^1\to L^\infty$ bound for $\mR_H^\pm V \mR_V^\pm V \mR_H^\pm$. } Lemma 6.4 in \cite{egd} and Lemmas~\ref{lem:high2} and \ref{lem:high1} below.  
\begin{lemma}\label{lem:high2}
Under the conditions of Proposition~\ref{prop:2highgamma}, we have
$$\sup_{x,y} \big|I\big(\chi_j(z)  \mR_L^{\pm}V\mR_V^{\pm}   V   \mR_H^{\pm}\big)\big|\les  \min(2^{2j},   2^{7j/2} |t|^{-1}).
$$
\end{lemma} 
\begin{proof} Using the resolvent identity, we write
\be\label{eq:high resid}
 \mR_L^{\pm}V\mR_V^{\pm}   V   \mR_H^{\pm}=\mR_L^{\pm}V\mR_0^{\pm}   V   \mR_H^{\pm}-\mR_L^{\pm}V\mR_0^{\pm} V\mR_V^{\pm}  V   \mR_H^{\pm}.
 \ee
 To bound the second summand without the time decay, we use  a limiting absorption principle for the perturbed resolvent operator of the form:
\be\label{eqn:lap}
\sup_{|\lambda|>\lambda_0}\|\partial_\lambda^k \mR_V^\pm (\lambda)\|_{L^{2,\sigma}\to L^{2,-\sigma}} \les 1,
\qquad \sigma > \f12 +k, \,\,\,\,k=0,1,... 
\ee  
holds for any $\lambda_0>m$.  This was proved in  \cite{egg} for $k=0$; the case $k>0$ follows from this and the resolvent identity.  We note that by equations \eqref{eqRL} and \eqref{eqRH}, we may write the resolvent in the middle as
\begin{align}\label{eq:R0 high2}
	\mR_0^\pm(z)(x,y)=\frac{i\alpha\cdot(x-y)}{2\pi|x-y|^2}+ \widetilde O_1\big(z^{1/2}(|x-y|^{-1/2}+|x-y|^{1/2})\big).
\end{align}
Then,
\begin{multline*}
	\int_{z\approx 2^j} \frac{z}{\sqrt{z^2+m^2}} \| \mathcal R_L^{\pm}V\mathcal R_0^{\pm}V\|_{L^{2,\sigma}} \|\mR_V^\pm \|_{L^{2,\sigma}\to L^{2,-\sigma}} \| V\mathcal R_H^{\pm}\|_{L^{2,\sigma}}\, dz\\
	\les 	\int_{z\approx 2^j} \frac{z^{2}}{\sqrt{z^2+m^2}}\, dz \les 2^{2j}. 
\end{multline*}
Here we use \eqref{eqRL} and \eqref{eq:R0 high2} to see
%
$$
	|\mR_L^\pm V\mR_0^\pm V|\les \int_{\R^2} \frac{z^{\f12}}{|x-x_1|}|V(x_1)|(|x_1-x_2|^{-1}+|x_1-x_2|^{\f12})|V(x_2)| \, dx_1.
$$
Therefore, by Lemma~\ref{lem:potential}  
\be\label{eq:RLVR0V}
\big\|\mR_L^\pm V\mR_0^\pm V\big\|_{L^{2,\f12+}_{x_2}}\les z^{\f12}
\ee
uniformly in $x$.  Similarly, for $V\mR_H^\pm$   the bound in \eqref{eqRH} implies that 
\be\label{eq:VRH}
\|V(y_1)\mR_H^\pm(z)(y_1,y)\|_{L^{2,\sigma}_{y_1}}\les z^{1/2}\la y\ra^{-1/2}.
\ee

  We now turn to the time decay.  We employ the stationary phase bound in Lemma~\ref{stat phase general} below by writing 
$$
 I(\mR_L^{\pm}V\mR_0^{\pm} V\mR_V^{\pm}  V   \mR_H^{\pm}) =	 \int_0^\infty e^{-i 2^{-3j} t \phi(z)}   a(z,x,y) dz, $$
where
$$
\phi(z)= 2^{3j}\left(\sqrt{z^2+m^2}- z |y| /t\right),
$$
$$
	a(z,x,y)=\frac{z \chi_j(z)}{\sqrt{z^2+m^2}}[\mR_L^\pm V \mR_0^\pm V \mR_V^\pm V \mR_H^\pm](z)(x,y) e^{\mp iz|y|}  
$$
We choose $\phi(z)$ in this way so that the lower bound $1\leq \phi''(z)$, which is needed to apply Lemma~\ref{stat phase general},  holds on the support of $a(z,x,y)$.    It is also this stationary phase bound that necessitated our restriction to dyadic energy levels.
Note that the bound in \eqref{eqRH} implies that 
$$
\partial_z \big[\mR_H^\pm (z)(y_1,y) e^{\mp iz|y|}\big] = O(z^{1/2} \la y_1\ra |y-y_1|^{-1/2}).
$$
Using this, \eqref{eq:RLVR0V},  \eqref{eq:VRH}, and a similar bound for  $\partial_z(\mR_L^\pm V\mR_0^\pm V)$, we obtain
$$
|a(z,x,y)|+|\partial_z a(z,x,y)|\les 2^j \chi_j(z) \la y\ra^{-1/2}.
$$ 
By Lemma~\ref{stat phase}, we estimate the integral above by
$$
\int_{|z-z_0|<\sqrt{2^{3j}/t }} |a(z)|\, dz 
    		+t^{-1} 2^{3j} \int_{|z-z_0|>\sqrt{2^{3j}/t }} \bigg( \frac{|a(z)|}{|z-z_0|^2}+
    		\frac{|a'(z)|}{|z-z_0|}\bigg)\, dz,
$$
where $z_0=m\frac{ |y|}{\sqrt{t^2- |y|^2}}$. In the case when $z_0$ is in a small neighborhood of the support of $a(z,x,y)$ we must have  $t\approx |y|$. Therefore, in this case, we have the bound
$$
2^j  \la y\ra^{-1/2} \left(\sqrt{2^{3j}/t }+t^{-1} 2^{3j} \frac{2^j}{\sqrt{2^{3j}/t }}\right)
\les 2^{7j/2}/t.
$$
In the case $t\not \approx |y|$, we have 
$$
\left|\partial_z\left(\sqrt{z^2+m^2}- z |y| /t\right)\right| \gtrsim 1.
$$
An integration by parts together with the bounds on $a(z,x,y)$ imply that the integral is bounded by
$2^{2j}/t$.  
  
The proof for the first summand in \eqref{eq:high resid} is  similar. 
 \end{proof} 
\begin{lemma}\label{lem:high1}
Under the conditions of Proposition~\ref{prop:2highgamma}, we have
$$\sup_{x,y}\big|I\big(\chi_j(z)  \mR_L^{\pm}V\mR_V^{\pm}   V   \mR_L^{\pm}\big)\big|\les  \min(2^{2j},   2^{7j/2} |t|^{-1}).
$$
\end{lemma} 
The proof of this lemma is similar but simpler since $\mR_L^\pm$ has no oscillatory part. By the  resolvent identity we write 
$$
 \mR_L^{\pm}V\mR_V^{\pm}   V   \mR_L^{\pm}=\mR_L^{\pm}V\mR_0^{\pm}   V   \mR_L^{\pm}-\mR_L^{\pm}V\mR_0^{\pm} V\mR_0^{\pm}   V   \mR_L^{\pm}+\mR_L^{\pm}V\mR_0^{\pm}V\mR_V^{\pm} V\mR_0^{\pm}  V   \mR_L^{\pm}.
 $$
The bound \eqref{eq:RLVR0V} and a similar one for the $z$ derivative  suffice to control each of these terms via an integration by parts.

\section{Weighted dispersive decay estimates} \label{sec:weighted}
In this section we show the Dirac evolution can decay faster in time as an operator between weighted spaces. As in Section~\ref{sec:nonweight}, we divide the proof  into two subsections. In the first subsection we show the statement of Theorem~\ref{th:main2} for small energies, in the support of $\chi(z)$.  In the second subsection we show the statement holds for large energies, in the support of the cut-off $\widetilde{\chi}(z)=1-\chi(z)$ without the need to restrict to dyadic energy levels.

\subsection{Small energy  weighted estimates}
In this section we will show that 
\begin{theorem}\label{thm:Sec4}
	Let $|V(x)| \les \la x \ra^{-5-}$. Then, we have
	for $t>2$
	\be\label{weighteddecay}
		 \bigg| \int_0^\infty e^{-it\sqrt{z^2+m^2}} \frac{ z\chi(z)}{\sqrt{z^2+m^2}}
		 [\mR_V^+(z)-\mR_V^-(z)](x,y) dz \bigg|
		 \les \f{w(x)w(y)}{t\log^2(t)}+
		\f{\la x\ra^{\f32} \la y\ra^{\f32}}{t^{1+}}
	\ee
	where $w(x) = 1+ \log^{+}|x| $.
	\end{theorem}
Using the symmetric resolvent identity as in Section~\ref{sec:nonweight}, we have
\begin{align} \label{Rv+-}
 \mR^{+}_V(z)-\mR^{-}_V(z) = [\mR^{+}_0 -\mR^{-}_0] - [\mR^{+}_0v^*M_{+}^{-1} v\mR^{+}_0 -\mR^{-}_0v^*M_{-}^{-1}v \mR^{-}_0] .
 \end{align}
We start with the contribution of the free resolvent. To establish the time decay we employ the following oscillatory integral bounds.

\begin{lemma} \label{lem:ibp} Let $\mathcal{E}(z) $ be supported on the neighborhood $(0,z_0)$ for some $z_0\ll 1$. Then,  for any $t>2$ we have
\begin{multline}
	\Big|  \int_0^\infty e^{-it\sqrt{z^2+m^2}} \frac{z}{\sqrt{z^2+m^2}} \mathcal{E}(z) dz -\frac{i e^{-imt}}{t}\mathcal{E}(0) \Big|\\
	\les {\f 1 t} \int_0^{t^{-1/2}} | \mathcal{E}^{\prime}(z) | dz  
	+\frac{1}{t^2} \int_{t^{-1/2}}^{\infty} \Big| \frac{\mathcal{E}^{\prime}(z)}{z^2} +\frac{ \mathcal{E}^{\prime \prime} (z)}{z}\Big| dz.
 \end{multline}
\end{lemma}

\begin{proof}
Using the identity $\frac{z}{\sqrt{z^2+m^2}} e^{-it\sqrt{z^2+m^2}}=-\frac{1}{it} \partial_z (e^{-it\sqrt{z^2+m^2}})$ to integrate by parts we have  
\begin{multline*}
\Big|\int_0^\infty e^{-it\sqrt{z^2+m^2}} \frac{z}{\sqrt{z^2+m^2}} \mathcal{E}(z) dz- \frac{ie^{-imt}}{t}\mathcal{E}(0) \Big| \les \\
   {\f 1 t} \int_0^{t^{-1/2}} | \mathcal{E}^\prime (z) |dz +{\f 1 t}\Big| \int_{t^{-1/2}}^{\infty} e^{-it\sqrt{z^2+m^2}}   \mathcal{E}^\prime (z) dz\Big|.
\end{multline*}
Applying another integration by parts to the last term on the right side 
\begin{multline*}
\Big|  \int_{t^{-1/2}}^{\infty} e^{-it\sqrt{z^2+m^2}}   \mathcal{E}^\prime (z) dz \Big| \\ \les \frac{1}{t^2} \Big( \frac{ \sqrt{z^2 + m^2}} {z} |\mathcal{E}^\prime(z)| \Big) \Big|_{z=t^{-1/2}}^{\infty}  
 + \frac{1}{t^2} \int_{t^{-1/2}}^{\infty} \big| \partial_z \big\{ \frac{\sqrt{z^2 + m^2}} {z} \mathcal{E}^{\prime}(z) \big\} \big| dz\\ \les 
 \frac{1}{t^2} \int_{t^{-1/2}}^{\infty} \big| \partial_z \big\{ \frac{\sqrt{z^2 + m^2}} {z} \mathcal{E}^{\prime}(z) \big\} \big| dz.
 \end{multline*}
 The last inequality follows from the support condition on $\mathcal E(z)=0$ for $z>1$ and the Fundamental Theorem of Calculus.
Finally note that, on the support of $\mathcal E(z)$, we have
$$
\Big|  \frac  { \sqrt{z^2 + m^2}} {z}  \Big| \les z^{-1},   \,\,\,\,\ \text{and} \,\,\,\,\ \Big|   \partial_z \Big\{ \frac {\sqrt{z^2 + m^2}} {z} \Big\}\Big| \les z^{-2},
$$
which yields the claim.
\end{proof} 
 \begin{corollary} \label{cor:logt}
If $\mathcal{E}(z) =\widetilde O_2(\log^{-2} (z)) $, then for $t>2$ we have 
\begin{align*}
\bigg| \int_0^\infty e^{-it\sqrt{z^2+m^2}} \frac{z\chi(z)}{\sqrt{z^2+m^2}} \mathcal{E}(z) dz \bigg|\les \frac{1}{t \log^2 t}. 
\end{align*}
\end{corollary}
\begin{proof} By Lemma~\ref{lem:ibp} it is enough to see that 
\begin{align}\label{eq:IBP1}
\f1t \int_{0}^{t^{-1/2}} \frac{1}{z \log^3 z} dz + \frac{1}{t^2} \int_{t^{-1/2}}^{\infty}  \frac{\chi(z)}{z^3 \log^3 z} dz \les \frac{1}{t \log^2 t}.
\end{align}
The desired bound for the first summand follows by direct calculation. To see the bound for the second summand, note that  
\begin{align*}
\frac{1}{t^2}\int_{t^{-1/2}}^{\infty}\frac{\chi(z) }{z^3 \log^3 z} dz   = \frac{1}{t^2}\int_{t^{-1/2}}^{t^{-1/4}} \frac{1}{z^3 \log^3 z} dz +\frac{1}{t^2} \int_{t^{-1/4}}^{z_0} \frac{1}{z^3 \log^3 z} dz.
\end{align*}
Hence, the bounds 
\begin{align*}
&\frac{1}{t^2}\bigg| \int_{t^{-1/4}}^{z_0} \frac{1}{z^3 \log^3 z} \bigg| \les \frac{1}{t^2}\bigg| \int_{t^{-1/4}}^{z_0} \frac{1}{z^3 } \bigg| \les t^{-3/2}, \\
& \frac{1}{t^2}\bigg|  \int_{t^{-1/2}}^{t^{-1/4}}  \frac{1}{z^3 \log^3 z} dz \bigg| \les  \frac{1} {t^2\log ^3 t}  \bigg| \int_{t^{-1/2}}^{t^{-1/4}} z^{-3} dz \bigg| \les \frac{1} {t\log ^3 t}, 
\end{align*} 
establish  the statement.
\end{proof}
The following lemma gives the contribution of the free resolvent to \eqref{weighteddecay}.

\begin{lemma} \label{lem:free} We have 
\begin{multline*}
\int_0^\infty e^{-it\sqrt{z^2+m^2}} \frac{ z\chi(z)}{\sqrt{z^2+m^2}}
		 [\mR_0^+(z)-\mR_0^-(z)](x,y) dz\\ = -\frac{m e^{-itm}}{t}M_{11}+ O\Big(\frac{ \la x \ra ^{3/2} \la y \ra ^{3/2} }{t^{\f54}} \Big).
\end{multline*}
\end{lemma}
\begin{proof}
We use Lemma~\ref{lem:ibp} for $\mathcal{E}(z) = \chi(z) [\mR^{+}_0- \mR^{-}_0](z)(x,y)$.  By \eqref{R0dif} we have 
\begin{align}  \label{withnozterm}
\chi(z) [\mR^{+}_0- \mR^{-}_0] & = \chi(z) mi M_{11}+ \chi(z) \widetilde O_2(z^{1/2} \la x \ra ^{3/2} \la y \ra ^{3/2}) 
\end{align}
 Therefore, $ \mathcal{E}(0) = mi M_{11}$ and $ |\partial_z^k\mathcal{E}(z) | \les z^{1/2-k} \la x \ra ^{3/2} \la y \ra ^{3/2}$ for $k=1,2$. Hence, by Lemma~\ref{lem:ibp} we obtain
\begin{multline*}
\Big|  \int_0^\infty e^{-it\sqrt{z^2+m^2}} \frac{z}{\sqrt{z^2+m^2}} \mathcal{E}(z) dz + \frac{m e^{-itm}}{t}M_{11} \Big|  \\ \les {\f {\la x \ra ^{3/2} \la y \ra ^{3/2}} t} \int_0^{t^{-1/2}} z^{-1/2}  dz   + \frac{\la x \ra ^{3/2} \la y \ra ^{3/2}}{t^2} \int_{t^{-1/2}}^{\infty} z^{-5/2} dz 
  \les t^{-5/4} \la x \ra ^{3/2} \la y \ra ^{3/2} .
 \end{multline*}
\end{proof}
Our approach for establishing Theorem~\ref{thm:Sec4} will be to control the integrals in \eqref{eq:ibp} directly.  Unlike in the proof of Theorem~\ref{th:main1}, we need to have the exact form of the boundary term at $z=0$ when integrating by parts.  These exact values are critical to our proofs, and hence our strategy will differ from the previous section. In addition, since we are considering bounds that can depend on $x,y$ our technical approach  and choice of expansions will differ. Using the expansion in Lemma~\ref{lem:R0exp}, \eqref{G0 def}, and expanding $\mathcal G_0$ into two terms, we write $\mR^{\pm}_0=\mR_6 + \mR_7^{\pm} $, where
\begin{align}\label{R2R1 def}
\begin{split}
&  \mR_6:= \frac{i \alpha(x-y)}{2\pi |x-y|^{2}},\\
&\mR_{7}^{\pm} :=2m g^{\pm}(z) M_{11} + 2mG_0 I_1 + E^{\pm}_0(z)(x,y).
\end{split}
\end{align} 
Here $E^{\pm}_0(z)(x,y)$ is not identical to the error term in Lemma~\ref{lem:R0exp}, however, it satisfies the same bounds. This is a slightly different decomposition than we use in Section~\ref{sec:nonweight}, in particular $\mR_6$ does not depend on $z$.

Now we consider the second  term in $\eqref{Rv+-}$.  Using the expansion above we write
\begin{multline}\label{weightedexp}
\mR_0^{\pm}v^*M^{-1}_\pm v\mR^{\pm}_0=\mR_6v^*M^{-1}_{\pm} v\mR_6  + \mR_7^{\pm}v^*M^{-1}_{\pm} v\mR_6 \\ 
+\mR_6 v^*M^{-1}_\pm v\mR_7^{\pm} + \mR_7^{\pm}v^*M^{-1}_\pm v\mR_7^{\pm} .
\end{multline}

We start estimating the contribution of the last term in the above sum to the Stone's formula, \eqref{eq:Stone}. Note that the boundary term appearing in Lemma~\ref{fromeg3}  cancels the boundary term appearing in Lemma~\ref{lem:free} above  when substituted into \eqref{Rv+-}.	   
\begin{lemma} \label{fromeg3} Let $|V(x)| \les \la x \ra ^{-5-}$. Then, for $t>2$ we have 
\begin{multline*} 
I\big([\mR_7^{+}v^*M_{+}^{-1}v \mR_7^{+}-\mR_7^{-}v^*M_{-}^{-1}v \mR_7^{-}](x,y)\big) \\
		   =-\frac{m e^{-itm}}{t}M_{11} + O\Big( \frac{ \big(1+\log^{+}|x| \big) \big(1+\log^{+}|y| \big)}{t\log^2 t} \Big) +O\Big(\frac{ \la x \ra ^{3/2} \la y \ra ^{3/2} }{t^{1+}} \Big).
\end{multline*}
\end{lemma}

\begin{proof}
We note that by \eqref{R2R1 def} and Lemma~\ref{lem:Minverseregular}, and recalling that $M_{11}v^*Q=QvM_{11}=0$, we have
\begin{multline*}
	\mR_7^\pm v^* M_{\pm}^{-1}v\mR_7^\pm (z)
	=\frac { [2m g^{\pm}(z) ]^2}{ h^{\pm}(z) }M_{11}v^*SvM_{11}+h_\pm^{-1}(z)[ G_0I_1v^* S v G_0 I_1 ]\\
	+\frac{g^{\pm}(z) }{ h^{\pm}(z) }[M_{11}v^* S v G_0 I_1 + G_0 I_1v^* S v M_{11}] +[G_0I_1 v^* QD_0QvG_0I_1] \\+  \widetilde O_2(z^{1/2-})\Omega_0(x,y).
\end{multline*}
Where $|\Omega_0(x,y)|\les (\la x\ra \la y \ra)^{3/2}$.
The contribution of the $\log^{-}|x-x_1|$ terms in $E_0^\pm$, see \eqref{R2R1 def}, to $\Omega_0$ are easily controlled in the $x_1$ integral since $ v(\cdot)\log^{-}|x-\cdot|\in L^2$.

 Recall that 
$ h^\pm (z) = (2m   g^\pm(z)+p)\| (a,c) \|^2$. Also note that using \eqref{g def} we have $g^+(z)-g^-(z)=\frac{i}{2}$. Thus we obtain 
\begin{align} \label{eq:hdecays}
	&\frac { [2m g^{+}(z) ]^2}{ h^{+}(z) }  - \frac { [2m g^{-}(z) ]^2}{h^{-}(z)}  
	= \frac{mi}{\| (a,c) \| ^2}+\widetilde O_2\big( (\log z )^{-2} \big), \\
	&\frac{g^{+} (z)}{h^{+}(z)} - \frac{g^{+} (z)}{h^{+}(z)} , \quad
	\frac{1}{h^+(z)}-\frac{1}{h^-(z)} =\widetilde O_2\big( (\log z )^{-2} \big), \nn
\end{align}
Now, recalling \eqref{G0 def} and the absolute boundedness of $S$, one can see $|G_0I_1v^*S vG_0I_1(x,y)|\les (1+\log^+|x|)(1+\log^+|y|)$, and  $| [M_{11}v^* S v G_0 I_1 + G_0 I_1v^* S v M_{11}] | \les  (1+\log^+|x|)(1+\log^+|y|).$ Hence, 
\begin{multline}\label{eq:R7MR7}
	[\mR_7^+ v^* M_{+}^{-1}v\mR_7^+-\mR_7^- v^* M_{-}^{-1}v\mR_7^-](z)
	=\frac{mi}{\| (a,c) \| ^2}M_{11}v^*SvM_{11}\\
	+(1+\log^+|x|)(1+\log^+|y|)\widetilde O_2\bigg( \frac{1}{\log^2 z}\bigg)+(\la x\ra \la y \ra)^{3/2} \widetilde O_2(z^{1/2-}).
\end{multline}
The second summand is bounded  by $(t \log^2 t)^{-1} (1+\log^+|x|)(1+\log^+|y|) $ using Corollary~\ref{cor:logt}, while the third summand is bounded by $t^{-1-} (\la x\ra \la y \ra)^{3/2}$ using Lemma~\ref{lem:ibp}. For the first summand, by integration by parts we have 
\begin{multline}\label{eq:bdyterms}
\int_0^{\infty} e^{-it\sqrt{z^2+m^2}} \frac{ z\chi(z)}{\sqrt{z^2+m^2}} [M_{11} v^* Sv M_{11} ](x,y) = - \frac{e^{-itm}}{it} [M_{11} v^* Sv M_{11} ](x,y)  \\ + [M_{11} v^* Sv M_{11} ](x,y) O(t^{-2}).
\end{multline}
Recalling the definition of $P$ and $S$ in Lemma~\ref{lem:Minverseregular}, 
we have $M_{11}v^*SvM_{11}=M_{11} v^* Pv M_{11}=\| (a,c) \| ^2 M_{11}$, which cancels the $\|(a,c)\|^2$ in the denominator of the first summand in \eqref{eq:R7MR7}.  This calculation also implies that the second term is bounded by $t^{-2}$ uniformly  in $x,y$. 
\end{proof}

Next we estimate the contribution of the rest of the terms in \eqref{weightedexp} to the Stone's formula. By symmetry it is enough to consider the terms 
$\mR_6 v^*M^{-1}_{\pm} v \mR_6 $ and $  \mR_6v^*M^{-1}_{\pm} v \mR_7^{\pm}$. We start with $  \mR_6v^*M^{-1}_{\pm} v \mR_7^{\pm}$. 
\begin{lemma} \label{lem:nw12} Let $|V(x)| \les \la x \ra ^{-5-}$. Then, for $t>2$, we have  
\begin{multline*} 
I( [\mR_6 v^*M_{+}^{-1}v \mR_7^{+}-\mR_6 v^*M_{-}^{-1}v \mR_7^{-}](x,y) )
		   = O\Big( \frac{ \big(1+\log^{+}|y| \big)}{t\log^2 t} \Big) +O\Big(\frac{ \la y \ra ^{3/2} }{t^{1+}} \Big).
\end{multline*}
\end{lemma}
\begin{proof} As in the previous section, since $\mR_6$ doesn't map $L^1\to L^2$, we iterate resolvent identities, \eqref{FIT}, to write 
\begin{multline*}
[\mR_6 v^*M_{+}^{-1}v \mR_7^{+}-\mR_6 v^*M_{-}^{-1}v \mR_7^{-}] = \\
- \mR_6 V\big(\mR^{+}_7 - \mR ^{-}_7 \big) - \mR_6V\big( \mR_0^{+} v^{*}M_{+}^ {-1}v\mR^{+}_7-\mR_0^{-} v^{*}M_{-}^ {-1}v\mR^{-}_7\big) =: - \Gamma_7^1- \Gamma_7^2.
\end{multline*}
As in the previous lemma, we proceed via integration by parts.  Both terms will have a boundary term of size $t^{-1}$ when $z=0$.  As before, we show that these terms cancel, and the remaining terms decay faster for large $t$.

We start with $ \Gamma^1_7$. By \eqref{R0dif}, then using  Lemma~\ref{lem:potential} we obtain
\begin{align*}
 \Gamma^1_7 &= mi \mR_6VM_{11} +  \widetilde O_2(z^{1/2}) \int_{\R^4}  [\mR_6V](x,x_1)\la x_1-y\ra^{3/2} dx_1  \\
 & = mi \mR_6VM_{11} + \widetilde O_2(z^{1/2}) \la y \ra ^{3/2} .
\end{align*} 
Using that $\mR_6$ is independent of $z$ and \eqref{eq:bdyterms} for the first summand, and using Lemma~\ref{lem:ibp} for the second, we see
$$
I(  {\Gamma}_7^1)= - \frac{me^{-imt}}{t} \mR_6VM_{11} + O(t^{-1-})\la y \ra ^{3/2}.
$$

Next we estimate $ \Gamma_7^2$. Using Lemma~\ref{lem:R0exp}, Lemma~\ref{lem:Minverseregular}, and recalling that $M_{11}v^*Q=QvM_{11}=0$, and \eqref{eq:hdecays},  we have
\begin{multline*}
	\Gamma_7^2(z)(x,y)\\
	=-\frac{mi}{\| (a,c) \| ^2} \mR_6V M_{11}v^*SvM_{11} + \widetilde O_2\bigg( \frac{1}{\log^2 z}\bigg) \Omega_1(x,y) + \widetilde O_2(z^{1/2-}) \Omega_2 (x,y),
\end{multline*}
where 
$$
\Omega_1  =   \mR_6V M_{11}v^*SvM_{11} +  \mR_6V M_{11}v^*SvG_0I_1 +  \mR_6V \mathcal{G}_0 v^*SvM_{11} +  \mR_6 V\mathcal{G}_0 v^*S vG_0I_1, 
$$
and the kernel of $\Omega_2$ (with $r_1= |x-x_1|$, $r_2 = | x_1- y_1|$, $r_3 = | y- y_1|$) satisfies the bound 
$$
|\Omega_2(x,y)|  \les \int_{\R^6} r_1^{-1} |V(x_1)| \big( r_2^{-1}  + r_2^{3/2} \big) |[v^* A v](x_2,y_1)| \big( r_3^{0-}  + r_3^{3/2}\big)  dx_1 dx_2 dy_1, 
$$
for some  absolutely bounded operator, $A$. Using Lemma~\ref{lem:potential}, one can see that 
$$
|\Omega_1(x,y) | \les 1+\log^{+}(y)  , \,\,\,\,\,\  |\Omega_2(x,y) | \les  \la y \ra^{3/2}.
$$
Hence, 
\begin{multline*}
	\Gamma_7^2(z) 
	=-\frac{mi}{\| (a,c) \| ^2} \mR_6V M_{11}v^*SvM_{11}  + (1+\log^{+} |y|)  \widetilde O_2\bigg( \frac{1}{\log^2 z}\bigg)  
	 + \la y \ra^{3/2} \widetilde O_2(z^{1/2-}).
\end{multline*}
Using \eqref{eq:bdyterms}, the first summand's contribution is
\begin{align*}
-\frac{mi}{\| (a,c) \| ^2} I(  \mR_6V M_{11}v^*SvM_{11}) & =   \frac{m e^{-imt} }{t \| (a,c) \| ^2} \mR_6 V M_{11}v^*SvM_{11}+ O(t^{-2}) \\
& =  \frac{m e^{-imt} }{t }\mR_6 V M_{11}+ O(t^{-2}).
\end{align*} 
Here, we once again used that $M_{11}v^*SvM_{11}=\| (a,c) \| ^2 M_{11}$.
The second summand in $\Gamma_7^2$ is bounded by $  \big(1+\log^{+}|y| \big)( t\log^2 t)^{-1} $ using Corollary~\ref{cor:logt}. The third summand is bounded by $ \la y \ra ^{3/2} t^{-1-} $ using Lemma~\ref{lem:ibp}. Adding $  I( {\Gamma}_6^1)$ to $I(\Gamma_7^2)$ completes the proof.
\end{proof}

\begin{lemma} Let $|V(x)| \les \la x \ra ^{-5-}$. Then, for $t>2$, we have  
\begin{align*} 
I\big( [\mR_6 v^*M_{+}^{-1}v \mR_6 -\mR_6 v^*M_{-}^{-1}v \mR_6 ](x,y) \big)
		   = O( t^{-1}(\log t)^{-2}).
\end{align*}
\end{lemma}
\begin{proof} Using the iteration formula \eqref{SIT} we  consider 
\begin{multline*}
 - \mR_6V[\mR_0^+- \mR_0^{-}] V \mR_6 +  \mR_6 V (\mR_0V M_{+}^ {-1} v\mR_0^{+}-\mR_0^{-} v^{*} M_{-}^ {-1}v\mR_0^{-}\big) V  \mR_6 =:   {\Gamma}_8^1 + \Gamma_8^2.
\end{multline*}
As in the Lemma~\ref{lem:nw12} we estimate $ {\Gamma}_8^1$ and $ \Gamma_8^2$ separately, and show that the leading order $t^{-1}$ terms cancel. 
By \eqref{R0dif} and  Lemma~\ref{lem:potential} we obtain
\begin{align*}
{\Gamma}_8^1 & = -mi \mR_6VM_{11}V \mR_6 +  \widetilde O_2(z^{1/2}) \int_{\R^4}  \mR_6 V\la x_1-y_1\ra ^{\f32}V \mR_6 dx_1 dy_1  \\
 & = -mi \mR_6VM_{11}V \mR_6 + \widetilde O_2(z^{1/2}).
\end{align*}
This gives 
$$
I( {\Gamma}_8^1) =  \frac{me^{-imt}}{t} \mR_6VM_{11}V \mR_6+ O(t^{-1-})
$$
by \eqref{eq:bdyterms} and Lemma~\ref{lem:ibp}. 

As in the proof of Lemma~\ref{lem:nw12},  using Lemma~\ref{lem:R0exp}, Lemma~\ref{lem:Minverseregular}, and recalling that $M_{11}v^*Q=QvM_{11}=0$, and \eqref{eq:hdecays},  we have
\begin{multline*}
	\Gamma_8^2 (z)(x,y)
	=\frac{mi}{\| (a,c) \| ^2} \mR_6V M_{11}v^*SvM_{11}V  \mR_6 \\ + \widetilde O_2\bigg( \frac{1}{\log^2 z}\bigg) \Omega_3(x,y) + \widetilde O_2(z^{1/2-}) \Omega_4 (x,y).
\end{multline*}
where 
\begin{multline*}
\Omega_3  =   \mR_6V M_{11}v^*SvM_{11} V \mR_6+   \mR_6V M_{11}v^*Sv \mathcal{G}_0 V \mR_6 +  \mR_6V \mathcal{G}_0 v^*SvM_{11}V \mR_6 \\
  +   \mR_6V \mathcal{G}_0 v^*Sv \mathcal{G}_0 V \mR_6,  
  \end{multline*}
 and the kernel $\Omega_4(x,y)$ is bounded by
\begin{align*}
	\int_{\R^6} r_1^{-1} |V(x_1)| \big( r_2^{-1}  + r_2^{3/2} \big) |[v^*Av]|(x_2,y_2)  \big( r_3^{-1}  + r_3^{3/2} \big) |V(y_2)| r_4^{-1}  dx_1 dx_2 dy_1 dy_2. 
\end{align*}
By Lemma~\ref{lem:potential}, one can see that $|\Omega_3 (x,y)| , |\Omega_4 (x,y) | \les 1 $ uniformly in $x$ and $y$.  So, 
\begin{align*}
	\Gamma_8^2(z)(x,y)
	=\frac{mi}{\| (a,c) \| ^2} \mR_6V M_{11}v^*SvM_{11}V \mR_6 
	    +  \widetilde O_2\bigg( \frac{1}{\log^2 z}\bigg).
\end{align*}
Thus, \eqref{eq:bdyterms} along with the equality $M_{11}v^*SvM_{11}=\| (a,c) \| ^2 M_{11}$ and 
Corollary~\ref{cor:logt} shows that
$$
	I(\Gamma_8^2) =-\frac{me^{-imt}}{t} \mR_6VM_{11}V \mR_6+ O(t^{-1}(\log t)^{-2}).
$$
Adding $   I( \Gamma_8^1)$ to $I(\Gamma_8^2)$ completes the proof.
\end{proof} 

\subsection{Large energy weighted estimates}

In this section we investigate the perturbed Dirac evolution at energies separated from the threshold. We prove the following proposition which implies Theorem~\ref{th:main2} for large energy.  In contrast to the argument in Section~\ref{sec:nonweight}, we do not localize to dyadic frequency intervals nor employ a specialized stationary phase argument.  Since the desired bound allows for a dependence on $x$ and $y$, we use the same expansions for the perturbed resolvent for the high energy argument in Section~\ref{sec:nonweight}, and obtain the desired time decay by integrating by parts sufficiently many times.
\begin{prop}\label{prop:high energy}
	Under the assumptions of Theorem~\ref{th:main2}
	we have the following bound for $t \geq1$,
	\be \label{eqn: high}
		\sup_{L\geq1} \bigg|\int_0^\infty e^{-it\sqrt{z^2+m^2}}\frac{z^{-2-} \tilde{\chi}(z)}{\sqrt{z^2+m^2}}  \chi(z/L)[\mathcal R_V^{+}(\lambda)- \mathcal R_V^{-}] (x,y)  dz\bigg|\\ \les  \frac{\la x \ra ^{\f 32} \la y \ra ^{\f32}}{t^{2}}
	\ee
	provided the components of $V$ satisfy the bound
	$|V_{ij}(x)|\les \la x\ra^{-5-}$.
\end{prop}
 As in previous sections we will estimate the contribution of the terms appearing in  the resolvent expansion \eqref{exp:resolvent} to \eqref{eqn: high} in a series of lemmas. For the convenience of the reader we recall \eqref{exp:resolvent}:
$$
		\mathcal R_V^{\pm}(\lambda) =  \mathcal R_0^{\pm}(\lambda)- \mathcal R_0^{\pm}(\lambda) V\mathcal R_0^{\pm}(\lambda)  + \mathcal R_0^{\pm}(\lambda)  V\mR_V^{\pm} (\lambda) V    \mR_0^{\pm}(\lambda).
$$ 
 We first note that the contribution of the first term in \eqref{exp:resolvent} can be handled similarly to Lemma~\ref{lem:free}. Specifically, we consider 
$$
	\int_0^\infty e^{-it\sqrt{z^2+m^2}} \frac{z}{\sqrt{z^2+m^2}} \mathcal E(z)\, dz
$$ 
with
 $\mathcal{E}(z) =z^{-3-} \tilde\chi(z) [\mR^{+}_0- \mR^{-}_0](z)(x,y)$.  The ample $z$ decay allows us to take $L=\infty$ for the cut-off $\chi(z/L)$.  By \eqref{R0dif} we have 
$$
   z^{-3-}\tilde\chi(z)  [\mR^{+}_0- \mR^{-}_0](z)(x,y)  = z^{-3-} \tilde\chi(z)  mi M_{11}+  \tilde\chi(z)  \widetilde O_2(z^{-\f52-} \la x \ra ^{3/2} \la y \ra ^{3/2}) 
$$
 Since the cut-off functions in $\mathcal E(z)$ are not supported at zero we  can integrate by parts twice without boundary terms and bound the contribution by $\frac{ \la x \ra ^{3/2} \la y \ra ^{3/2}}{t^{2}}$.
 
We next consider the contribution of the second and third terms in \eqref{exp:resolvent}. We start with the second term. Recalling $\mR_0^\pm= \mR_L^\pm +\mR_H^\pm $, it suffices to consider the contributions of 
$$\Gamma_1:=(\mR^+_L-\mR^-_L)V\mR_L^+,\,\,\,\,\,\Gamma_2:=\mR^+_LV\mR^+_H,\,\,\,\,\,\,\Gamma_3:=\mR^+_HV\mR^+_H, $$
that arise when substituting \eqref{exp:resolvent} into \eqref{eqn: high}.
\begin{lemma}\label{lem:first three}
Under the assumptions of Theorem~\ref{th:main2} the following bound holds for each $k=1,2,3$.
\begin{align}
\sup_{L\geq1} \bigg|\int_0^\infty e^{-it\sqrt{z^2+m^2}}\frac{z^{-2-} \tilde{\chi}(z)}{\sqrt{z^2+m^2}}  \chi(z/L) \Gamma_k (x,y)  dz\bigg|  \les  \frac{\la x \ra ^{\f 32} \la y \ra ^{\f32}}{t^{2}}.
\end{align}
\end{lemma}
\begin{proof} For each $\Gamma_j$ we will integrate by parts twice. We start with $\Gamma_1$. By the relationship \eqref{eq:Dmpm}, Lemma~\ref{lem:schro_resol} and the fact that $\sqrt{z^2+m^2}=\widetilde O(z)$ when $z\gtrsim 1$, one has
\begin{align}\label{eq:secondrl}
	\mR_L^\pm(z)(x,x_1) = \frac{i}{2\pi} \frac{\alpha \cdot (x-x_1)}{|x-x_1|^2}+  \widetilde O_2( z(z|x-y|^{-1+})).
\end{align}
Hence, by Lemma~\ref{lem:potential2}, 
\begin{align*}
	\Gamma_1(z)(x,y) = \widetilde O_2(z^{\f12+})  \int _{\R^2} \frac{\la x_1 \ra^{-5-}} {|x-x_1|^{1-}|x_1-y|} dx_1  = \widetilde O_2( z^{\f12+}).  
\end{align*}
By integration by parts twice, noting the lack of boundary terms due to the cut-offs,   we obtain
\begin{multline} \label{tibp}
\bigg| \int_0^\infty e^{-it\sqrt{z^2+m^2}}\frac{z^{-2-} \tilde{\chi}(z)}{\sqrt{z^2+m^2}}  \chi(z/L) \Gamma_1 (x,y)  dz \bigg|  \\
 \les {\f 1 {t^2}} \int_{z_0}^{\infty} \big| \partial_z \Big\{\frac{\sqrt{z^2+m^2}}{z} \partial_z\{z^{-3-}\Gamma_1(z)\}(x,y)  \Big\}  \big| dz \les \frac{ \la y \ra ^{3/2}}{t^2}.
\end{multline}
For $ \Gamma_2 $ and $ \Gamma_3 $ we note that, using $
\mR_H^\pm(z)(x,y) = e^{\pm i z|x-y|}\widetilde w_\pm(z|x-y|)$, one has for $z\gtrsim1$ 
\begin{align}\label{eq:secondrh}
 |\partial_z^k \{\mR_H^\pm(z)(x,y)\} | \les z^{1/2} (|x-x_1|^{-{1/2}} + |x-x_1|^{3/2}) \text { for } k=0,1,2.
\end{align}
Therefore, Lemma~\ref{lem:potential2} together with \eqref{eq:secondrh} and \eqref{eq:secondrl} gives  ($r_1= |x-x_1|$, $r_2= |x_1-y|$)
\begin{align*}
 |\partial_z^k \{\Gamma_2(z)(x,y)\} | & \les z^{\f12}  \int _{\R^2} \frac{ r_1^{-1} ( r_2^{-{1/2}} + r_2^{3/2})}{ \la x_1 \ra^{5+}} dx_1 \les z^{\f12} \la y \ra ^{3/2}, \\
 |\partial_z^k \{\Gamma_3 (z)(x,y)\} |&  \les z  \int _{\R^2} \frac{( r_1^{-{1/2}} +r_1^{3/2}) ( r_2^{-{1/2}} + r_2^{3/2})}{ \la x_1 \ra^{5+}} dx_1 \les z \la x \ra ^{3/2} \la y \ra ^{3/2},
\end{align*}
for $k=0,1,2$. Integration by parts twice gives the statement as in \eqref{tibp}. 
\end{proof} 

\begin{lemma}\label{lem:wtdhigh2}
	Under the assumptions of Theorem~\ref{th:main2} we have
	\begin{align*} 
	\sup_{L\geq1} \bigg|\int_0^\infty e^{-it\sqrt{z^2+m^2}}\frac{z^{-2-} \tilde{\chi}(z)}{\sqrt{z^2+m^2}}  \chi(\lambda/L) [\mathcal R_0^{\pm} V \mathcal R_V^{\pm}V    \mathcal R_0^{\pm}](x,y)  dz\bigg|  \les  \frac{\la x \ra ^{\f 32} \la y \ra ^{\f32}}{t^{2}}.
	\end{align*}
\end{lemma} 
\begin{proof} We drop the $\pm$ signs in this proof.
By  the resolvent identity we have
$$
\mR_0 V \mR_V V\mR_0= \mR_0 V \mR_0 V\mR_0 - \mR_0 V\mR_0 V \mR_0 V\mR_0 +  \mR_0 V\mR_0 V\mR_V V  \mR_0 V\mR_0.
$$
The contribution of the first two terms to \eqref{eqn: high} can be estimated by $\la x \ra ^{3/2} \la y \ra ^{3/2} t^{-2}$ as in the Lemma~\ref{lem:first three}, noticing that by  \eqref{eq:secondrl}  and \eqref{eq:secondrh}  one has 
\begin{align} \label{secondr0}
|\partial_z^k \{\mR_0 (z)(x,y)\} | \les z^{1/2} (|x-x_1|^{-1} + |x-x_1|^{3/2}) \text { for $ k=0,1,2 $ },
\end{align}
for $z\gtrsim1$. 
Using these bounds in Lemma~\ref{lem:potential2} with $r_1= |x-x_1|$, $r_2= |x_1-y_1|$, $r_3= |y_1-y|$, we obtain for $k=0,1,2$
\begin{multline*}
 |\partial_z^k \{[\mR_0 V \mR_0 V\mR_0](z)(x,y)\} | \les z^{3/2}  \int _{\R^4} \frac{( r_1^{-1} +r_1^{3/2}) ( r_2^{-1} + r_2^{3/2})( r_3^{-1} + r_3^{3/2})}{ \la x_1 \ra^{5+}\la y_1 \ra^{5+}} dx_1 dy_1 \\ 
  \les z^{3/2} \la x \ra ^{3/2} \la y \ra ^{3/2}, 
   \end{multline*}
Similarly (with $r_1= |x-x_1|$, $r_2= |x_1-x_2|$, $r_3= |x_2-y_1|$, $r_4= |y_1-y|$.)
$$
 |\partial_z^k \{[\mR_0 V \mR_0 V\mR_0V\mR_0](z)(x,y)\} | 
  \les z^2 \la x \ra ^{3/2} \la y \ra ^{3/2}. 
$$
Hence, integration by parts twice establishes the desired bound. 

  Finally, we will prove the statement for the term containing perturbed resolvent and establish Proposition~\ref{prop:high energy} . In order to control this term we recall \eqref{eqn:lap},  
$$
\| \partial_z^k \mR _V(z) \|_{ L^{2,\sigma}\rightarrow L^{2,-\sigma} } \les 1, \,\,\, \sigma > {\f 12}+k, \,\, k=0,1,2,
$$
for $z\gtrsim 1$.

Using the bound \eqref{secondr0} in Lemma~\ref{lem:potential}  one can obtain
\begin{multline*}
\| \partial_z^k \{[\mathcal R_0  V\mathcal R_0 V](x,x_2) \} \|_{L^2_{x_2}} \\ \les z  \Big \| V(x_2) \int _{\R^2} \frac{ (r_1^{-1} + r_1^{3/2})   (r_2^{-1} + r_2^{3/2})}{ \la x_1 \ra ^{5+} } dx_1\Big\|_{L^{2,\sigma}_{x_2}}  \les z \la x \ra^{3/2},  
\end{multline*}
where $k=0,1,2$, $\sigma = \f32+ $, $r_1=|x-x_1|$ and $r_2=|x_1-x_2|$. Therefore, limiting absorption principle gives (with $k_j\geq 0$ and $k_1+k_2+k_3=0,1,2$)
\begin{align*}
| [\partial_z^{k_1} \{\mathcal R_0  V\mathcal R_0 \} V  \partial_z^{k_2} \{ \mathcal R_V \}V \partial_z^{k_1}\{ \mathcal R_0 V \mathcal R_0 \}](x,y) | \les z^2 \la x \ra ^{3/2}\la y \ra ^{3/2},
\end{align*}
Integrating by parts as in \eqref{tibp} finishes the proof. 
\end{proof}

\begin{rmk}\label{rmk:wtdhigh}
	
	Since our goal is to use the bound we have just proven in Theorem~\ref{th:main2} and interpolate with Theorem~\ref{th:main1}, we need not pursue optimal smoothness of the initial data.  We note that, 
	one can prove a bound that is sharper with respect to derivative loss but with larger spatial weights.  This may be achieved by writing $\mR_0=\mR_L+\mR_H$ and iterating resolvent identities only $\mR_L^\pm$ in the proof of Lemma~\ref{lem:wtdhigh2} as in the unweighted bound of Proposition~\ref{prop:2highgamma}.  In particular, under the hypotheses of Theorem~\ref{th:main}, 
	\begin{align*}
	\big\| \la \cdot \ra ^{-\f32} e^{-itH} P_{ac}(H)  \la H \ra ^{-2-} f \big\|_{L^{\infty}(\R^2)} \les \frac{1}{ |t| \log^2 |t|} \big\| \la \cdot \ra^{\f32}  f \big\|_{L^1(\R^2)}, \,\,\ |t|>2.
	\end{align*} 
	
\end{rmk}

\section{Spatial bounds and stationary phase estimates}\label{sec:tech}
In this section we state several technical lemmas that were used throughout the paper. 
\begin{lemma} \label{lem:potential} Let $\beta >\max\{2,2p+2\}$. Then we have
$$
\Big \| \int _{\R^2} (1+|x-x_1|^{-1}) \la x_1 \ra ^{-\beta} |x_1 - x_2|^{p} \la x_2\ra ^{-\beta/2} dx_1\Big\|_{L^2_{x_2}} \les 1
$$
for $p\geq -1$. 
\end{lemma}
To prove Lemma~\ref{lem:potential}, we use the following  estimate from \cite{EG1}.
\begin{lemma} \label{lem:potential2} Fix $u_1$, $u_2 \in \R^n$ and let $ 0 \leq k$, $l <n$ , $\beta >0$ , $k+l+\beta \geq n$ , $k+l \neq n $. We have
        $$ \int_{\R^n} \frac{\la x \ra ^{-\beta-}}{ |x- u_1|^k |x-u_2|^l} dx \les \left\{
            \begin{array}{ll}
            (\frac{1}{|u_1-u_2|})^{\max(0, k+l-n)} & \quad  |u_1-u_2| \leq 1, \\
             (\frac{1}{|u_1-u_2|})^{\min(k, l, k+l+\beta -n)} & \quad |u_1-u_2| > 1.
        \end{array} \right .
        $$
\end{lemma}   
\begin{proof}[Proof of Lemma~\ref{lem:potential}]
We first consider $-1\leq p\leq 0$.  In this case, if $p>-1$ we use Lemma~\ref{lem:potential2} in the $x_1$ integral to see
$$	
\int _{\R^2} (1+|x-x_1|^{-1}) \la x_1 \ra ^{-\beta} |x_1 - x_2|^{p} \la x_2\ra ^{-\beta/2} dx_1 \les \la x_2\ra^{-\beta/2} \la x-x_2\ra^{p} \les \la x_2\ra^{-\beta/2} \in L^2_{x_2}.
$$	
On the other hand, if $p=-1$, we use
$$
	\frac{1}{|x-x_1|\, |x_1-x_2|}\les \frac{1}{|x-x_1|}\bigg(\frac{1}{|x-x_1|^{1-}}+\frac{1}{|x_1-x_2|^{1+}}
	 \bigg)
$$
In which case, we use that Lemma~\ref{lem:potential2} and the fact that $(1+|x-x_2|^{0-}) \la x_2\ra^{-\beta/2}\in L^2_{x_2}$.  If $p\geq 0$, we may reduce to the $p=0$ case by noting $|x_1-x_2|^p \les \la x_1\ra^p \la x_2 \ra^p$, which necessitates the larger value of $\beta$.
\end{proof}

We recall Lemma~3.5 in \cite{egd}.

\begin{lemma}\label{lem:high stat phase2}   If 
	$$
	|a(z)|\les 	\frac{z\chi_j(z) \widetilde \chi(zr)}{(1+zr)^{\f12}} ,
	\qquad |\partial_z a(z)|\les 
	\frac{\chi_j(z) \widetilde \chi(zr)}{(1+zr)^{\f12}},
	$$
	then we have the bound
	\begin{align*}
	\bigg|\int_0^\infty e^{-it \phi_\pm(z) } a(z)\, dz\bigg| \les  \min(2^{2j}, 2^{\frac{3j}2}|t|^{-1/2}, 2^{2j} |t|^{-1}),
	\end{align*}
	where $\phi_\pm(z)=\sqrt{z^2+m^2}\mp \frac{z r}{t}$.
\end{lemma}

We have  the  following (slightly modified) lemma from
\cite{Sc2}, see \cite[Lemma 3.3]{egd}
\begin{lemma}\label{stat phase general}

	Let $\phi'(z_0)=0$ and $1\leq \phi'' \leq C$.  Then,
\begin{multline*}
    \bigg| \int_{-\infty}^{\infty} e^{-it\phi(z)} a(z)\, dz \bigg|
    \lesssim \int_{|z-z_0|<|t|^{-\frac{1}{2}}} |a(z)|\, dz \\
    +|t|^{-1} \int_{|z-z_0|>|t|^{-\frac{1}{2}}} \bigg( \frac{|a(z)|}{|z-z_0|^2}+
    		\frac{|a'(z)|}{|z-z_0|}\bigg)\, dz.
\end{multline*}

\end{lemma}

\end{document}